\documentclass[12pt,reqno]{amsart}

\usepackage{amsfonts,amssymb}

\usepackage{graphicx}
\usepackage{epsfig}
\usepackage{enumitem}
\usepackage{latexsym}
\usepackage{amsmath,amsthm}
\usepackage{mathrsfs}
\usepackage[all,cmtip]{xy}

\usepackage[pagebackref=true]{hyperref}
\hypersetup{   pdfborder={0 0 0},   
              colorlinks = false}
\usepackage{tikz-cd}

\renewcommand*{\backref}[1]{}
\renewcommand*{\backrefalt}[4]{%
  \ifcase #1 %
    No citations.
  \or
    ($\uparrow$ #4).%
  \else
    ($\uparrow$ #4).%
  \fi%
}

\theoremstyle{plain}
\newtheorem{theorem}{Theorem}[section]
\newtheorem{lemma}[theorem]{Lemma}
\newtheorem{definition}[theorem]{Definition}
\newtheorem{proposition}[theorem]{Proposition}
\newtheorem{cor}[theorem]{Corollary}
\newtheorem{remark}[theorem]{Remark}

\numberwithin{equation}{section}

\newcommand{\pushright}[1]{\ifmeasuring@#1\else\omit\hfill$\displaystyle#1$\fi\ignorespaces}
\newcommand{\pushleft}[1]{\ifmeasuring@#1\else\omit$\displaystyle#1$\hfill\fi\ignorespaces}
\newcommand{\trans}[2]{\tau_{#1#2}}
\newcommand{\cocycle}[3]{g_{#1#2}^{#3}}
\newcommand{\conjcocycle}[3]{\overline{g}_{#1#2}^{#3}}
\newcommand{\dholos}[1]{\mathfrak{#1}}

\newcommand{\bigslant}[2]{{\raisebox{.2em}{$#1$}\left/\raisebox{-.2em}{$#2$}\right.}}

\hypersetup{colorlinks, citecolor=blue, filecolor=black, linkcolor=red}
\begin{document}
\baselineskip=15.5pt

\title{Hodge decomposition theorem on compact $d$-K\"ahler manifolds}
\author{Sanjay~Amrutiya}
\address{Department of Mathematics, IIT Gandhinagar,
 Near Village Palaj, Gandhinagar - 382355, India}
 \email{samrutiya@iitgn.ac.in}
\author{Ayush~Jaiswal}
\address{Department of Mathematics, Indian Institute of Science Education and 
Research (IISER) Tirupati, Andhra Pradesh-517507, India}
\email{ayushjaiswal@labs.iisertirupati.ac.in; ayushjwl.math@gmail.com}
\subjclass[2000]{Primary: 53C56; Secondary: 53C25, 53C07}
\keywords{Hodge decompoition; $d$-complex manifold}
\thanks{SA is supported by the SERB-DST under project no. CRG/2023/000477. AJ is supported by the 
Institute Post-Doctoral Fellowship (IISER-T/Offer/PDRF(M)/A.J/04/2023) by IISER Tirupati.}
\date{}

\begin{abstract}
In this article, we will explore the fundamental concepts, including various basic concepts on 
$d$-complex manifolds, along with several differential operators and examine the relationships 
between them. A $d$-K\"ahler manifold is a $d$-complex manifold equipped with a metric that 
satisfies a specific condition. We prove the Hodge decomposition theorem on compact $d$-K\"ahler 
manifolds, which establishes a crucial relationship between certain de-Rham cohomology groups and 
Dolbeault cohomology groups on a compact $d$-K\"ahler manifold .

\end{abstract}
\maketitle

\section{Introduction}
The Hodge decomposition theorem is a cornerstone in complex differential geometry, 
providing insights into the interplay between topology, analysis, and algebraic geometry. This 
article establishes an analogous theorem on compact d-K\"ahler manifolds, exploring various 
fundamental concepts on $d$-complex manifolds. 

In Section \ref{section2}, we will review the prerequisites, including the definitions of $d$-complex 
manifolds. This involves studying $d$-complex vector bundles, almost $d$-complex structure, 
$d$-complex smooth forms, and $d$-Hermitian metric.
In Section \ref{section3}, we will introduce several differential operators and examine their relationships. 
These operators will play a pivotal role in both computations and proofs throughout the article.
For a deeper analysis of these operators, the Appendix \ref{appendix:sectionA} provides the necessary 
framework using certain Sobolev spaces. The treatment in Appendix is adapted from \cite[Chapter 4]{ROW} 
tailored for $d$-complex manifolds. Section \ref{section4} begins with the concepts of $d$-K\"ahler manifold
and Hodge $\star$-operator, where the analysis is particularly simplified. After preparing this groundwork,
we present the Hodge decomposition theorem on a compact $d$-K\"ahler manifold. The basic examples of 
$d$-K\"ahler manifolds are Klein surfaces. Using the Hodge decomposition theorem, we also give examples of 
$d$-complex manifolds that do not admit $d$-K\"ahler structure. The last section discusses 
$d$-holomorphic connections on $d$-holomorphic vector bundles. More specifically, we prove that if 
$d$-holomorphic connection exists on a $d$-holomorphic vector bundle $E$, then all the Chern classes of 
$E$ vanishes. This fact was asserted in \cite[Remark 22]{SAAJ}.

\subsection*{Notations}
\begin{center}
\begin{tabular}{p{2.5cm}p{0.5cm}p{13cm}}
$(X,\dholos{X})$ &=& Smooth manifold $X$ with $d$-complex structure $\dholos{X}$.\\ & & \\

$(X,\mathcal{J})$ &=& Smooth manifold $X$ with almost $d$-complex structure $\mathcal{J}$.\\ &&\\

$\pi:E\rightarrow X$ &=& Smooth $d$-complex vector bundle on $d$-complex manifold\\ & & \\

$\pi:\mathcal{E}\rightarrow X$ &=& $d$-holomorphic vector bundle on $d$-complex manifold $(X,\dholos{X})$.\\ && \\
$\mathcal{U}$ &=& $d$-holomorphic atlas $\mathcal{U}=\{(U_i,\phi_i)\}_{i\in I}$ on $X$\\  & & \\

$\mathcal{U}^E$ &=& Trivialization for $d$-complex bundle $E$, $\mathcal{U}^E=\{(U_i,\psi_i)\}_{i\in I}$\\  & & \\

$\cocycle{j}{i}{E}$ &=& Co-cycle map for the bundle $E$\\ & & \\

$\trans{j}{i}$ &=& Transition map $\trans{j}{i}:\phi_i(U_i\cap U_j)\rightarrow \phi_j(U_i\cap U_j)$ \\ & & \\

$\trans{j}{i}^E$ &=& Transition map $\trans{j}{i}^E=(\trans{j}{i},\cocycle{j}{i}{E} \circ \phi_i^{-1}):\phi_i(U_i\cap U_j)\rightarrow \phi_j(U_i\cap U_j)\times \mathrm{GL}(n,\mathbb{C})\text{ or, }\mathrm{GL}(n,\mathbb{R})$ \\ & & \\

$T_X$, $T^\star_X$ &=& The tangent and co-tangent bundle for the smooth manifold associated to $(X,\dholos{X})$.\\ & & \\

$T^\mathcal{C}_X$, ${T^\star}^\mathcal{C}_X$ &=& The smooth $d$-complex tangent and co-tangent bundle for $d$-complex\\
&& manifold $(X,\dholos{X})$.\\ & & \\

$D(\trans{j}{i})$ &=& Jacobian of the map $\trans{j}{i}$ i.e. $D(\trans{j}{i})=\frac{\partial(x_j^1,y_j^1,x_j^2,y_j^2,\dots,x_j^n,y_j^n)}{\partial(x_j^1,y_j^1,x_i^2,y_i^2,\dots,x_i^n,y_i^n)}$.\\ & & \\

$T^{(p,q)}(X)$ &=& The space of smooth $d$-complex vector fields of type $(p,q)$.\\ & & \\

${T^\star}^{(p,q)}(X)$ &=& The space of smooth $d$-complex forms of type $(p,q)$.\\ & & \\

$\mathcal{A}^p(X)$ &=& The space of smooth $p$-forms on the associated smooth manifold $X$.\\ & & \\

$\mathcal{A}^p(E)(X)$ &=& The space of smooth $p$-forms on the associated smooth manifold taking values in bundle $E$ on $X$.\\ & & \\

$\mathcal{A}^p_\mathcal{C}(X)$ &=& The space of $d$-complex smooth $p$-forms on $(X,\dholos{X})$.\\ & & \\

$\mathcal{A}^{(p,q)}(X)$ &=& The space of $d$-complex smooth $(p,q)$-forms on $(X,\dholos{X})$.\\ & & \\

$H^p_{dR}(X)$ &=& The de Rham cohomology group of smooth manifold $X$.\\ 
\end{tabular}
\end{center}
\begin{center}
\begin{tabular}{p{2.5cm}p{0.5cm}p{13cm}}
$H^p_{dB}(X)$ &=& The Dolbeault cohomology group of $d$-complex manifold $(X,\dholos{X})$.\\ &&\\

$\mathrm{Diff_m(E, F)}$ &=& The space of differential operators between $d$-complex bundles $E$ and $F$. \\ && \\

$\mathrm{PDiff_m(E, F)}$ &=& The space of pseudo-differential operators between $d$-complex bundles $E$ and $F$.\\ && \\

$\mathrm{OP_m(E, F)}$ &=&the differential operators of order $m$, which have continuous extension to sobolev completions.\\ && \\

$\mathcal{H}_L(E)$ &=& The kernel of an operator $L$, which is subset of $E(X)$. \\ & & \\

$\mathcal{A}^0_{\mathcal{C},k}$ &=& The sheaf of $d$-complex $k$-times differentiable functions on $(X,\dholos{X})$.\\ & & \\

$\big(U\rightarrow \mathcal{D}(U)\big)$ &=& Sheaf of smooth function with compact support on $X$.\\ && \\

$\mathcal{O}_X^{dh}$ &=& Sheaf of $d$-holomorphic functions on $d$-complex manifold $(X,\dholos{X})$.\\
\end{tabular}
\end{center}
\section{$d$-complex and almost $d$-complex manifolds}\label{section2}
In this section, we begin by reviewing the concepts of $d$-complex manifolds, almost $d$-complex 
structures, differential forms, and $d$-Hermitian metrics on $d$-complex vector bundles.
See \cite{SW1} for a more general notion of twisted complex manifolds.

Let $X$ be a topological smooth real manifold of even dimension $2n$ (for some positive integer $n$),
which may be orientable or non-orientable. A $d$-holomorphic atlas is a collection
$\mathcal{U}={(U_i,\phi_i)}_{i\in I}$, where $\phi_i: U_i\rightarrow V_i\subset\mathbb{C}^n$ is a 
homeomorphism, such that the transition map 
$$
\trans{j}{i}=\phi_j\circ\phi_i^{-1}:\phi_i(U_i\cap U_j)\rightarrow \phi_j(U_i\cap U_j)\qquad (i,j\in I)
$$
is either holomorphic or anti-holomorphic on each connected component of $\phi_i(U_i\cap U_j)$ and $\displaystyle\bigcup_{i\in I}U_i=X$. The unique maximal $d$-holomorphic atlas containing 
$\mathcal{U}={(U_i,\phi_i)}_{i\in I}$, denoted by $\dholos{X}$, is referred to as a $d$-complex structure. 
The pair $(X,\dholos{X})$ is then called a $d$-complex manifold.

From now on, take $(X,\dholos{X})$ be a $d$-complex manifold with $d$-holomorphic atlas $\mathcal{U}=\{(U_i,\phi_i)\}_{i\in I}$.
\begin{definition}\rm{
A smooth $d$-complex vector bundle $E$ on a $d$-complex manifold $(X,\dholos{X})$ is a family of complex trivial vector bundles 
$\{U_i\times\mathbb{C}^k\}_{i\in I}$, varying smoothly, and for 
$(x_j,\xi_j)\in \phi_j(U_j\cap U_i)\times \mathbb{C}^n\overset{\psi_j}\cong E|_{U_j\cap U_i}$, $(x_i,\xi_i)\in \phi_i(U_i\cap U_j)\times\mathbb{C}^n\overset{\psi_i}\cong E|_{U_i\cap U_j}$, the gluing condition is
\begin{align*}
        x_j&=\trans{j}{i}(x_i)\hspace{98pt} \text{   and }\\   
        \xi_j(x_j)&= \left\{\begin{array}{lr}
        \cocycle{j}{i}{E}\circ \phi_i^{-1}(x_i)\xi_i(x_i), & \text{if } \trans{j}{i} \text{ is holomorphic};\vspace{7pt}\\
        \cocycle{j}{i}{E}\circ\phi_i^{-1}(x_i)\overline{\xi}_i(x_i), & \text{if } \trans{j}{i} \text{ is anti-holomorphic}.\\
        \end{array}\right.
 \end{align*}
 
The co-cycle map $\cocycle{j}{i}{E}(x)\in \mathrm{GL}_n(\mathbb{C})$ (respectively, 
$\mathrm{GL}_n(\mathbb{R})$) for $x\in U_i\cap U_j$ (respectively, $x\in\partial(U_i\cap U_j)$) is 
smooth and satisfies the following cocycle condition: 
 \[
        \cocycle{k}{i}{E}(x)= \left\{\begin{array}{lr}
        \cocycle{k}{j}{E}(x)\circ\cocycle{j}{i}{E}(x), & \text{if } \trans{k}{j} \text{ is holomorphic};\vspace{7pt}\\
        \cocycle{k}{j}{E}(x)\circ\conjcocycle{j}{i}{E}(x), & \text{if } \trans{k}{j} \text{ is anti-holomorphic},\\
        \end{array}\right.\quad\text{ for }x\in U_i\cap U_j\cap U_k
      \]
along with
\[
        \cocycle{j}{i}{E}(x)= \left\{\begin{array}{lr}
        {\cocycle{i}{j}{E}}(x)^{-1}, & \text{if } \trans{j}{i} \text{ is holomorphic};\vspace{7pt}\\
        {\conjcocycle{i}{j}{E}}(x)^{-1}, & \text{if } \trans{j}{i} \text{ is anti-holomorphic}.\\
        \end{array}\right.\quad\text{for }x\in U_i\cap U_j
      \]
}
\end{definition}

The \emph{trivial $d$-complex line bundle} $\mathcal{C}$ on $(X,\dholos{X})$ is defined by the transition functions 
$\cocycle{j}{i}{\mathcal{C}} = 1$ with trivilization as $d$-holomorphic atlas $\mathcal{U}$.

A vector bundle $E$ is called a $d$-holomorphic vector bundle if the cocycle map 
$$
\cocycle{j}{i}{E}\circ \phi_i^{-1}:\phi_i(U_i\cap U_j)\rightarrow GL(n,\mathbb{C})~\mbox{(or}, GL(n,\mathbb{R}))
$$ 
is either holomorphic or anti-holomorphic, depending on whether the transition map 
$$
\trans{j}{i}:\phi_i(U_i\cap U_j)\rightarrow \phi_j(U_i\cap U_j)
$$ is holomorphic or anti-holomorphic on each connected component of $\phi_i(U_i\cap U_j)$ for $i,j\in I$.

Let $E$ be a smooth $d$-complex vector bundle on a $d$-complex manifold $(X,\dholos{X})$. A $d$-holomorphic 
structure on $E$ is a maximal family of pairs $\{(U_i,\psi_i)\}_{i\in I}$ such that 
$E|_{U_i}\overset{\psi_i}\cong \phi_i(U_i)\times \mathbb{C}^r$ ($i\in I$) and the transition map $\trans{j}{i}^E=(\trans{j}{i}, \cocycle{j}{i}{E}\circ \phi_i^{-1})$ is either holomorphic or anti-holomorphic on each connected component of $\phi_i(U_i\cap U_j)$.

\begin{remark}\rm{
From now on, for simplicity, we will assume that if intersection $U_i\cap U_j$ is non-empty, it is connected, for 
given two charts $(U_i,\phi_i)$ and $(U_j,\phi_j)$ $(i,j\in I)$ on the $d$-complex manifold $(X,\dholos{X})$. 
}
\end{remark}

\subsection{Almost $d$-complex structure}\label{subsec2.1}
Let $(X,\dholos{X})$ be a $d$-complex manifold of complex dimension $n$, the associated smooth manifold 
will be of real dimension $2n$. Let $T_X$ be the smooth tangent bundle for the smooth manifold $X$. 
The $d$-holomorphic atlas $\mathcal{U}=\{(U_i,\phi_i)\}_{i\in I}$ will also be a smooth atlas on the associated smooth manifold as holomorphic and anti-holomorphic maps are already smooth maps. 
For a chart $\big(U_p,\phi_p=(x_p^1,y_p^1,x_p^2,y_p^2,\dots,x_p^n,y_p^n)\big)$ around a point $p\in X$, the fiber $T_X(x)$ ($x\in U_p$) is a real vector space of 
dimension $2n$, which has a canonical complex structure:
$$
J_{U_p}:\mathbb{R}^{2n}\rightarrow \mathbb{R}^{2n}
$$
given by 
$J_{U_p}\big(x^1_p(x),y^1_p(x),\dots,x^n_p(x),y^n_p(x)\big)=\big(-y^1_p(x),x^1_p(x),\dots,-y^n_p(x),x^n_p(x)\big)$.

Define a line bundle $\mathcal{L}$ on $d$-complex manifold $(X,\dholos{X})$ with same trivializing cover as $d$-holomorphic atlas $\mathcal{U}$ has, with co-cycle map
\[
    \cocycle{j}{i}{\mathcal{L}} = \left\{\begin{array}{lr}
        1, & \text{if } \trans{j}{i} \text{ is holomorphic};\vspace{7pt}\\
        -1, & \text{if } \trans{j}{i} \text{ is anti-holomorphic}
        \end{array}\right.
  \]
then, we have a family of complex structures 
$\{J_{U_i}:T_X(U_i)\rightarrow T_X(U_i)\}_{i\in I}$ with respect to the atlas $\mathcal{U}=\{(U_i,\phi_i)\}_{i\in I}$ 
such that $J_{U_i}^2\equiv -id_{T(U_i)}$, satisfying the compatibility condition:
\begin{equation}\label{almostcompati0}
J_{U_j}\circ \left[D(\trans{j}{i})\right]= \left\{\begin{array}{lr}
        \left[D(\trans{j}{i})\right]\circ J_{U_i}, & \text{if } \cocycle{j}{i}{\mathcal{L}}=1;\vspace{7pt}\\
        -\left[D(\trans{j}{i})\right]\circ J_{U_i}, & \text{if } \cocycle{j}{i}{\mathcal{L}}=-1.\\
        \end{array}\right.\quad\text{here }D(\trans{j}{i})=\text{Jacobian of }\trans{j}{i}
\end{equation}
The unique maximal family containing $\{J_{U_i}\}_{i\in I}$ such that $J_{U_i}^2\equiv -id_{T_X(U_i)}$ and satisfying the condition \eqref{almostcompati0}
is called the canonical almost $d$-complex structure, $\mathcal{J}$ on the associated smooth manifold.

\begin{definition}\rm{
For a given smooth manifold $X$ of real dimension $2n$ (for some $n\in \mathbb{Z}^+$) with smooth atlas $\mathcal{U}=\{\big(U_i,\phi_i=(x_i^1,y_i^1,\dots,x_i^n,y_i^n)\big)\}_{i\in I}$ and orientation line bundle $\mathcal{L}$ with tirivialing cover same as the atlas $\mathcal{U}$ has with co-cycle map $\cocycle{j}{i}{\mathcal{L}}:U_i\cap U_j\rightarrow \{1,-1\}$, a maximal family  of morphisms $\{J_{U_i}:T_X(U_i)\rightarrow T_X(U_i)\}_{i\in I}$ such that $J_{U_i}^2 = -id_{T_X(U_i)}$ satisfying the compatibility condition
\begin{equation}\label{almostcompati}
J_{U_j}\circ \left[D(\trans{j}{i})\right]= \left\{\begin{array}{lr}
        \left[D(\trans{j}{i})\right]\circ J_{U_i}, & \text{if } \cocycle{j}{i}{\mathcal{L}} = 1;\vspace{7pt}\\
        -\left[D(\trans{j}{i})\right]\circ J_{U_i}, & \text{if } \cocycle{j}{i}{\mathcal{L}} = -1,\\
        \end{array}\right.
\end{equation} 
is called an almost $d$-complex structure, will be denoted by $\mathcal{J}$.
}
\end{definition}

From the above discussion, we have the following:
\begin{proposition}
A $d$-complex manifold $X$ induces a canonical almost $d$-complex structure on its underlying smooth 
real manifold.
\end{proposition}

\subsection{Differential forms}
Let $(X,\mathcal{J})$ be an almost $d$-complex manifold, where $X$ is a smooth manifold of even real dimension with orientation line bundle $\mathcal{L}$ having same trivialing cover as the smooth atlas $\mathcal{U}=\{(U_i,\phi_i)\}_{i\in I}$ has on $X$ 
and $T_X$ be the smooth tangent bundle.

For each $i\in I$, the complex structure $J_{U_i}:T_X(U_i)\rightarrow T_X(U_i)$ can be extended to a complex linear endomorphism on $T_X(U_i)\otimes_\mathbb{R}\mathbb{C}$, also will be denoted by $J_{U_i}$. This extension satisfies $J_{U_i}^2\equiv -id_{T_X(U_i)\otimes_\mathbb{R}\mathbb{C}}$ and the family $\{J_{U_i}\}_{i\in I}$ satisfies the compatibility condition \eqref{almostcompati}. Each $J_{U_i}$ has eigenvalues $i$ and 
$-i$ with corresponding eigenspaces denoted by $T^{(1,0)}(U_i)$ and $T^{(0,1)}(U_i)$, 
respectively. That is,
\begin{align*}
T^{(1,0)}(U_i)&=\{Z_{U_i}\in T_X(U_i)\otimes_\mathbb{R}\mathbb{C}:J_{U_i}(Z_{U_i})=iZ_{U_i}\}\quad\text{and}\\
T^{(0,1)}(U_i)&=\{Z_{U_i}\in T_X(U_i)\otimes_\mathbb{R}\mathbb{C}:J_{U_i}(Z_{U_i})=-iZ_{U_i}\}.
\end{align*}
By using the fact that any element $Z_{U_i}\in T_X(U_i)\otimes_\mathbb{R}\mathbb{C}$ can be expressed as 
$Z_{U_i}=Y^1_{U_i}+iY^2_{U_i}$ for some $Y^1_{U_i},Y^2_{U_i}\in T_X(U_i)$, it can be verified that
\begin{align*}
T^{(1,0)}(U_i)&=\{Y_{U_i}-iJ_{U_i}(Y_{U_i}):Y_{U_i}\in T_X(U_i)\}\\
T^{(0,1)}(U_i)&=\{Y_{U_i}+iJ_{U_i}(Y_{U_i}):Y_{U_i}\in T_X(U_i)\}
\end{align*}

From above discussion, we have the following:

\begin{proposition}
For an almost $d$-complex manifold $(X,\mathcal{J})$, an element $Z_{U_i}\in T_X(U_i)\otimes_\mathbb{R}\mathbb{C}$ 
is of type $(1,0)$ (respectively, of type $(0,1)$) that is, $Z_{U_i}\in T^{(1,0)}(U_i)$ (respectively, 
$Z_{U_i}\in T^{(0,1)}(U_i)$) if and only if $Z_{U_i}=Y_{U_i}-iJ_{U_i}(Y_{U_i})$ (respectively, 
$Z_{U_i}=Y_{U_i}+iJ_{U_i}(Y_{U_i})$) for some $Y_{U_i}\in T_X(U_i)$. Furthermore, for a given smooth vector 
field $Y=\{Y_{U_i}\}_{i\in I}\in T_X(X)$, there is a family $\{Z_{U_i}\in T^{(1,0)}(U_i)\}_{i\in I}$ (respectively, 
$\{Z_{U_i}\in T^{(0,1)}(U_i)\}_{i\in I}$) with respect to the smooth atlas $\mathcal{U}=\{(U_i,\phi_i)\}_{i\in I}$ 
on the smooth manifold $X$, satisfying the compatibility condition:
\begin{equation}\label{d-complexform}
Z_{U_j}= \left\{\begin{array}{lr}
        Z_{U_i}, & \text{if } \cocycle{j}{i}{\mathcal{L}}=1 ;\vspace{7pt}\\
        \overline{Z}_{U_i}, & \text{if } \cocycle{j}{i}{\mathcal{L}}=-1.\\
        \end{array}\right.
\end{equation}
\end{proposition}
\begin{proof}
The condition \eqref{d-complexform} follows from the compatibility condition \eqref{almostcompati}.
\end{proof}

The maximal family containing $\{Z_{U_i}\in T^{(1,0)}(U_i)\}_{i\in I}$ (respectively, $\{Z_{U_i}\in T^{(0,1)}(U_i)\}_{i\in I}$) is called a $d$-complex vector field of type $(1,0)$ (respectively, $(0,1)$) on the almost $d$-complex manifold $(X,\mathcal{J})$. The collection of smooth $d$-complex vector fields of type $(1,0)$ (respectively, $(0,1)$) on the 
almost $d$-complex manifold $(X,\mathcal{J})$ will be written as $T^{(1,0)}(X)$ (respectively, $T^{(0,1)}(X)$). 
An element of the space $\wedge^pT^{(1,0)}(X)\otimes \wedge^q T^{(0,1)}(X)$ is called a $d$-complex smooth 
vector field of type $(p,q)$. 

The direct sum $T^{(1,0)}(X)\oplus T^{(0,1)}(X)$ is the space of 
smooth $d$-complex vector $1$-fields on the almost $d$-complex manifold $(X,\mathcal{J})$, and is denoted by 
$T^\mathcal{C}_X(X)$.

Using the fact that a complex structure on a real vector space induces a complex structure on the dual vector space, we have an induced complex structure on $T^\star_X(U_i)$ which further induces a complex linear endomorphism on $T^\star_X(U_i) \otimes_\mathbb{R} \mathbb{C}$. 
Following a similar line of arguments as above, we have the following:

\begin{proposition}
For an almost $d$-complex manifold $(X,\mathcal{J})$, an element $Z_{U_i}\in T^\star_X(U_i)\otimes_\mathbb{R}\mathbb{C}$ is of type $(1,0)$ (respectively, 
$(0,1)$) that is, $Z_{U_i}\in {T^\star}^{(1,0)}(U_i)$ (respectively, $Z_{U_i}\in {T^\star}^{(0,1)}(U_i)$) 
if and only if $Z_{U_i}=Y_{U_i}-iJ_{U_i}(Y_{U_i})$ (respectively, $Z_{U_i}=Y_{U_i}+iJ_{U_i}(Y_{U_i})$) for 
some $Y_{U_i}\in T^\star_X(U_i)$ ($i\in I$). Furthermore, for a given smooth 1-form $Y=\{Y_{U_i}\}_{i\in I}\in T^*_X(X)$, there is 
a family $\{Z_{U_i}\in {T^\star}^{(1,0)}(U_i)\}_{i\in I}$ (respectively, 
$\{Z_{U_i}\in {T^\star}^{(0,1)}(U_i)\}_{i\in I}$) with respect to the smooth atlas 
$\mathcal{U}=\{(U_i,\phi_i)\}_{i\in I}$ on the smooth manifold $X$, satisfying the compatibility condition
\begin{equation}\label{d-complexform1}
Z_{U_j}= \left\{\begin{array}{lr}
        Z_{U_i}, & \text{if } \cocycle{j}{i}{\mathcal{L}}=1;\vspace{7pt}\\
        \overline{Z}_{U_i}, & \text{if } \cocycle{j}{i}{\mathcal{L}}=-1.\\
        \end{array}\right.
\end{equation}
\end{proposition}
\begin{proof}
The condition \eqref{d-complexform1} follows from the compatibility condition \eqref{almostcompati}.
\end{proof}

The maximal family containing $\{Z_{U_i}\in {T^\star}^{(1,0)}(U_i)\}_{i\in I}$ (respectively, $\{Z_{U_i}\in {T^\star}^{(0,1)}(U_i)\}_{i\in I}$) is called a $d$-complex vector form of type $(1,0)$ (respectively, of type $(0,1)$) on the almost $d$-complex manifold $(X,\mathcal{J})$. The collection of $d$-complex smooth forms of type $(1,0)$ (respectively, $(0,1)$) on the almost $d$-complex manifold 
$(X,\mathcal{J})$ will be written as ${T^\star}^{(1,0)}(X)$ (respectively, ${T^\star}^{(0,1)}(X)$). An 
element of the space $\wedge^p {T^\star}^{(1,0)}(X) \otimes \wedge^q {T^\star}^{(0,1)}(X)$ is called a 
$d$-complex smooth form of type $(p,q)$.

The direct sum ${T^\star}^{(1,0)}(X)\oplus{T^\star}^{(0,1)}(X)$ is the space of smooth $d$-complex 
vector $1$-forms on the almost $d$-complex manifold $(X,\mathcal{J})$, will be written as ${T^\star}^\mathcal{C}_X(X)$.

Each $d$-complex smooth $k$-form $\alpha=\{\alpha_{U_i}\}_{i\in I}\in \wedge^k{T^\star}^\mathcal{C}_X(X)$ 
decomposes as follows
$$
\alpha=\displaystyle\sum_{\substack{ p+q=k\\ p,q=1}}^k \alpha^{(p,q)}
$$
where $\alpha^{(p,q)}\in \wedge^p {T^\star}^{(1,0)}(X) \otimes \wedge^q {T^\star}^{(0,1)}(X)$.

From now on, the space of $d$-complex smooth $k$-forms (respectively, $d$-complex smooth $(p,q)$-forms) 
on the almost $d$-complex manifold $(X,\mathcal{J})$ will be denoted by $\mathcal{A}^k_\mathcal{C}(X)$ (respectively, 
$\mathcal{A}^{(p,q)}(X)$). We denote by $\mathcal{A}^k(E)(X)$ the space of smooth $k$-form on the smooth 
manifold $X$ taking values in a bundle $E$.

\begin{remark}\label{exterior-dc}\rm{
For an almost $d$-complex manifold $(X,\dholos{X})$, the line bundle $\mathcal{L}$ is a flat bundle and using the fact that a flat bundle has a canonical flat connection, we have extended de-Rham operator $d_\mathcal{L}:\mathcal{A}^0(\mathcal{L})(X)\rightarrow \mathcal{A}^1(\mathcal{L})(X)$. Also, there is a bijective corrrespondence between global sections of trivial $d$-complex line bundle $\mathcal{C}$ (see \cite{SW1}) and the bundle $\mathbb{R}\oplus \mathcal{L}$, define $d_\mathcal{C}(\omega)=\big(d(\omega^1),d_\mathcal{L}(\omega^2)\big)$ for a global section $\omega=\{\omega_{U_i}\}_{i\in I}\in \mathcal{C}(X)$. Furthermore, for a smooth $d$-complex function $f=\{f_{U_i}\}_{i\in I}\in \mathcal{C}(X)$, $d_\mathcal{C}(f)$ can be expressed as a family $\{d(f_{U_i})\}_{i\in I}$ satisfying
\[
    d(f_{U_j}) = \left\{\begin{array}{lr}
        d(f_{U_i}), & \text{if } \cocycle{j}{i}{\mathcal{L}}=1;\vspace{7pt}\\
        \overline{d(f_{U_i})}, & \text{if } \cocycle{j}{i}{\mathcal{L}}=-1
        \end{array}\right.
  \]
}
\end{remark}

\begin{proposition}
Let $(X,\mathcal{J})$ be an almost $d$-complex manifold. Then, we have
$$
d_\mathcal{C}(\mathcal{A}^{(p,q)}(X))\subset \mathcal{A}^{(p+2,q-1)}(X)\oplus \mathcal{A}^{(p+1,q)}(X)\oplus 
\mathcal{A}^{(p,q+1)}(X)\oplus \mathcal{A}^{(p-1,q+2)}(X).
$$
\end{proposition}
\begin{proof}
Any element $\alpha\in \mathcal{A}^{(p,q)}(X)$ can be written as 
$$\alpha=\alpha_0\wedge\alpha_1\wedge\dots\wedge\alpha_p\wedge\alpha_1'\wedge\dots\wedge\alpha_q',
$$
where $\alpha_0\in \mathcal{A}^{(0,0)}(X)$, $\alpha_i\in \mathcal{A}^{(1,0)}(X)$ $(1\leq i \leq p)$ 
and $\alpha_j'\in \mathcal{A}^{(0,1)}(X)$ $(1\leq j\leq q)$. Using the property of de-Rham operator, 
we have
\begin{align*}
d_\mathcal{C}(\mathcal{A}^{(0,0)}(X))&\subset \mathcal{A}^{(1,0)}(X)\oplus \mathcal{A}^{(0,1)}(X)\\
d_\mathcal{C}(\mathcal{A}^{(1,0)}(X))&\subset \mathcal{A}^{(2,0)}(X)\oplus \mathcal{A}^{(1,1)}(X)\oplus\mathcal{A}^{(0,2)}(X)\quad\text{and}\\
d_\mathcal{C}(\mathcal{A}^{(0,1)}(X))&\subset \mathcal{A}^{(2,0)}(X)\oplus\mathcal{A}^{(1,1)}(X)\oplus\mathcal{A}^{(0,2)}(X).
\end{align*}
Furthermore, we have
\begin{align*}
d_\mathcal{C}(\alpha)&=d_\mathcal{C}(\alpha_0)\wedge\alpha_1\wedge\dots\wedge\alpha_p\wedge\alpha_1'\wedge\dots\wedge\alpha_q'\\
&\qquad+\alpha_0\wedge d_\mathcal{C}(\alpha_1)\wedge\dots\wedge\alpha_p\wedge\alpha_1'\wedge\dots\wedge\alpha_q'\\
&\qquad\qquad\qquad\qquad \vdots\\
&\qquad\qquad+(-1)^{p+q-1}\alpha_0\wedge\alpha_1\wedge\dots\wedge\alpha_p\wedge\alpha_1'\wedge\dots\wedge d_\mathcal{C}(\alpha_q'),\quad\text{i.e., }\\
d_\mathcal{C}(\alpha)&\in \mathcal{A}^{(p+2,q-1)}(X)\oplus \mathcal{A}^{(p+1,q)}(X)\oplus \mathcal{A}^{(p,q+1)}(X)\oplus \mathcal{A}^{(p-1,q+2)}(X).
\end{align*}
Hence, we have the required inclusion
$$
d_\mathcal{C}(\mathcal{A}^{(p,q)}(X))\subset \mathcal{A}^{(p+2,q-1)}(X)\oplus \mathcal{A}^{(p+1,q)}(X)\oplus \mathcal{A}^{(p,q+1)}(X)\oplus \mathcal{A}^{(p-1,q+2)}(X).
$$
\end{proof}

\begin{definition}\rm{
An almost $d$-complex structure $\mathcal{J}$ on a smooth manifold $X$ of even dimension is said to be integrable 
if and only if
\begin{align*}
d_\mathcal{C}(\mathcal{A}^{(1,0)}(X))&\subset \mathcal{A}^{(2,0)}(X)\oplus \mathcal{A}^{(1,1)}(X) \text{ along with}\\
d_\mathcal{C}(\mathcal{A}^{(0,1)}(X))&\subset \mathcal{A}^{(1,1)}(X)\oplus \mathcal{A}^{(0,2)}(X).
\end{align*}
}
\end{definition}

Let $\mathcal{J}$ be an integrable almost $d$-complex structure on a smooth manifold $X$. We define 
$\partial := \pi^{(p+1,q)}\circ d_\mathcal{C}:\mathcal{A}^{(p,q)}(X)\rightarrow \mathcal{A}^{(p+1,q)}(X)$ and 
$\overline{\partial} := \pi^{(p,q+1)}\circ d_\mathcal{C}:\mathcal{A}^{(p,q)}(X)\rightarrow \mathcal{A}^{(p,q+1)}(X)$. 
Then, we have
$$
d_\mathcal{C}\omega=\partial \omega+\overline{\partial}\omega,
$$
for $\omega\in \mathcal{A}^{(p,q)}(X)$.

Considering the associated smooth manifold $X$ for a $d$-complex manifold $(X,\dholos{X})$. We have a sheaf of smooth $i$-forms ($U\rightarrow \mathcal{A}^i(U)$) on the smooth manifold $X$ and the corresponding de-Rham chain complex
$$
0\rightarrow \mathcal{A}^0(X)\xrightarrow{d}\mathcal{A}^1(X)\xrightarrow{d}\dots .
$$
The de-Rham cohomology of the associated smooth manifold $X$ is defined as
\begin{equation}
H^q_{dR}(X)=\frac{Ker(d:\mathcal{A}^q(X)\rightarrow\mathcal{A}^{q+1}(X))}{Im(d:\mathcal{A}^{q-1}(X)\rightarrow \mathcal{A}^q(X))}.
\end{equation}
Also, after taking the tensor product of chain complex of smooth forms on the manifold with $d$-complex line bundle $\mathcal{C}\rightarrow X$ (see \cite{SW1}), we have a chain complex of $d$-complex smooth forms
\begin{equation}\label{derhamcomp}
0\rightarrow \mathcal{A}^0_\mathcal{C}(X)\xrightarrow{d_\mathcal{C}}\mathcal{A}^1_\mathcal{C}(X)\xrightarrow{d_\mathcal{C}}\dots .
\end{equation}
and the associated de-Rham cohomology is defined as
\begin{equation}
H^q_{dR}(X,\mathcal{C})=\frac{Ker(d_\mathcal{C}:\mathcal{A}_\mathcal{C}^q(X)\rightarrow\mathcal{A}_\mathcal{C}^{q+1}(X))}{Im(d_\mathcal{C}:\mathcal{A}_\mathcal{C}^{q-1}(X)\rightarrow \mathcal{A}_\mathcal{C}^q(X))}.
\end{equation}
Using the canonical almost $d$-complex structure on the smooth manifold $X$ associated to a
$d$-complex manifold $(X,\dholos{X})$, we have the decompositions
\begin{align*}
\mathcal{A}^1_\mathcal{C}(X) &=\mathcal{A}^{(1,0)}(X)\oplus \mathcal{A}^{(0,1)}(X)\text{ and}\\
\mathcal{A}_\mathcal{C}^k(X)&=\bigoplus_{\substack{p+q=k \\ 1\leq p,q\leq k}}\mathcal{A}^{(p,q)}(X),
\end{align*}
where $\mathcal{A}^{(p,q)}(X)=\wedge^p{T^\star}^{(1,0)}(X)\otimes \wedge^q{T^\star}^{(0,1)}(X)$ and 
$\mathcal{A}_\mathcal{C}^k=\wedge^k{T^\star}^\mathcal{C}_X(X)$. We have the Dolbeault double complex
\[\begin{tikzcd}[ampersand replacement=\&]
	\vdots \& \vdots \& \vdots \\
	{\mathcal{A}^{(0,1)}(X)} \& {\mathcal{A}^{(1,1)}(X)} \& {\mathcal{A}^{(2,1)}(X)} \\
	{\mathcal{A}^{(0,0)}(X)} \& {\mathcal{A}^{(1,0)}(X)} \& {\mathcal{A}^{(2,0)}(X)}
	\arrow["{\overline{\partial}}", from=2-1, to=1-1]
	\arrow["{\overline{\partial}}", from=2-2, to=1-2]
	\arrow["{\overline{\partial}}", from=2-3, to=1-3]
	\arrow["{\overline{\partial}}", from=3-1, to=2-1]
	\arrow["{\overline{\partial}}", from=3-2, to=2-2]
	\arrow["{\overline{\partial}}", from=3-3, to=2-3]
	\arrow["\partial", from=2-1, to=2-2]
	\arrow["\partial", from=2-2, to=2-3]
	\arrow["\partial", from=3-2, to=3-3]
	\arrow["\partial", from=3-1, to=3-2].
\end{tikzcd}\]
along with Dolbeault chain complex
\begin{equation}\label{dolbeaultcomp}
0\rightarrow \mathcal{A}^{(p,0)}(X)\xrightarrow{\overline{\partial}}\mathcal{A}^{(p,1)}(X)\xrightarrow{\overline{\partial}}\mathcal{A}^{(p,2)}(X)\rightarrow\dots\quad(\text{for }0\leq p\leq n).
\end{equation}
The Dolbeault cohomology of a $d$-complex manifold $(X,\dholos{X})$ is defined by
\begin{equation}\label{Dolbcohomology}
H^{(p,q)}_{dB}(X)=\frac{Ker\big(\overline{\partial}:\mathcal{A}^{(p,q)}(X)\rightarrow\mathcal{A}^{(p,q+1)}(X)\big)}{Im\big(\overline{\partial}:\mathcal{A}^{(p,q-1)}(X)\rightarrow\mathcal{A}^{(p,q)}(X)\big)}.
\end{equation}

\subsection{$d$-Hermitian manifolds}
Let $(X,\dholos{X})$ be a $d$-complex manifold of dimension $n$, with $d$-holomorphic atlas 
$\mathcal{U}=\big\{\big(U_l,\phi_l=(z_l^1,z^2_l,\dots,z^n_l)\big)\big\}_{l\in I}$. 

\begin{definition}\rm{
A $d$-Hermitian metric on the $d$-complex manifold, is a Riemannian metric $h$, respecting the canonical almost $d$-complex structure $\mathcal{J}$ i.e.
$$
h(Y^1,Y^2)=h\big(\mathcal{J}(Y^1),\mathcal{J}(Y^2)\big)
$$ 
for $Y^1,Y^2\in T(X)$.
}
\end{definition}
An almost $d$-complex manifold with a Riemannian metric $g$ always has a Hermitian metric 
$h$, given by
$$
h(Y^1,Y^2)=g(Y^1,Y^2)+g(J(Y^1),J(Y^2))\quad(\text{for real vector fields }Y^1,Y^2)
$$

\begin{definition}\rm{
A $d$-complex manifold $(X,\dholos{X})$ with $d$-Hermitian metric $h$, is called a $d$-Hermitian manifold.
}
\end{definition}
The $d$-Hermitian metric $h$ on the $d$-complex manifold $(X,\dholos{X})$ can be expressed as a family of 
$C^\infty$ tensors $\left\lbrace\displaystyle\sum_{i,j=1}^nh^{U_l}_{ij}dz^i_l\otimes d\overline{z}^j_l\right\rbrace_{l\in I}$ with respect to the atlas $\mathcal{U}$, satisfying the compatibility condition 
\begin{equation}\label{hermicondi}
h^{U_{l'}}_{i'j'} = \left\{\begin{array}{lr}
        \displaystyle\sum_{ij}h^{U_l}_{ij}\frac{\partial z^i_l}{\partial z^{i'}_{l'}}\frac{\partial \overline{z}^j_l}{\partial \overline{z}^{j'}_{l'}}, & \text{if } \trans{l'}{l}\text{ is holomorphic};\vspace{7pt}\\
        \displaystyle\sum_{ij}\overline{h}^{U_l}_{ij}\frac{\partial \overline{z}^i_l}{\partial z^{i'}_{l'}}\frac{\partial z^j_l}{\partial \overline{z}^{j'}_{l'}}, & \text{if } \trans{l'}{l}\text{ is anti-holomorphic}.
        \end{array}\right.
\end{equation}
Here $h^{U_l}_{ij}\in C^\infty(U_l)$ $(l\in I)$ are real valued smooth functions satisfying
\begin{enumerate}
\item $h^{U_l}_{ij}\geq 0$\quad(positive definite)\vspace{7pt}
\item $\overline{h^{U_l}_{ij}}=h^{U_l}_{ji}$\quad(Hermitian symmetric)
\end{enumerate}

\begin{definition}\rm{
A $d$-Hermitian metric on a smooth $d$-complex bundle $E$, can be expressed as a family $\big\{\langle,\rangle_{U_i}:E(U_i)\times E(U_i)\rightarrow \mathbb{C}\big\}_{i\in I}$ of 
Hermitian inner products with respect to the atlas $\mathcal{U}=\{(U_i,\phi_i)\}_{i\in I}$, varying smoothly and satisfying the compatibility condition
\[
\langle,\rangle_{U_j} = \left\{\begin{array}{lr}
        \langle,\rangle_{U_i}, & \text{if } \trans{l'}{l}\text{ is holomorphic};\vspace{7pt}\\
        \overline{\langle,\rangle}_{U_i}, & \text{if } \trans{l'}{l}\text{ is anti-holomorphic}.
        \end{array}\right.
\]
}
\end{definition}

\begin{theorem}
Every smooth $d$-complex vector bundle on $d$-complex manifold $(X,\dholos{X})$ admits a $d$-Hermitian metric.
\end{theorem}
\begin{proof}
Let $\mathcal{U}=\{(U_i,\phi_i)\}_{i\in I}$ be a $d$-holomorphic atlas for a $d$-complex manifold $(X,\dholos{X})$  
such that the covering $\{U_i\}_{i\in I}$ of $X$ is locally finite and $E|_{U_i}$ is trivial complex vector bundle. 
Let $\{\rho_i\}_{i\in I}$ be a smooth partition of unity subordinate to the covering $\{U_i\}_{i\in I}$ and define a Hermitian metric $\tilde{h}^{U_i}$ on $E|_{U_i}$ by setting,
$$
\tilde{h}^{U_i}(\xi_{U_i},\eta_{U_i})(x)=\overline{^t\eta_{U_i}}(x).\xi_{U_i}(x),
$$
for $\xi_{U_i},\eta_{U_i}\in E(U_i)$. Here the multiplication is between co-ordinate vector associated to 
$\xi_{U_i}$ and $\eta_{U_i}$ with respect to a fixed frame of $E_{U_i}$.

For two global section $\xi=\{\xi_{U_i}\}_{i\in I}$ and $\eta=\{\eta_{U_i}\}_{i\in I}$ of the bundle $E$ along with considering $\trans{i}{j}$ (respectively, $\trans{i}{j'}$) is holomorphic (respectively, anti-holomorphic) for $j\in J\subset I\setminus \{i\}$ (respectively, $j'\in J'\subset I\setminus \{i\}$), we define
\begin{align*}
\langle \xi_{U_i},\eta_{U_i} \rangle_{U_i} 
& = h^{U_i}(\xi_{U_i},\eta_{U_i})\\
& := \rho_i\tilde{h}^{U_i}(\xi_{U_i},\eta_{U_i})+\sum_{j\in J\subset I\setminus \{i\}} 
\rho_j\tilde{h}^{U_j}(\xi_{U_j},\eta_{U_j})+\sum_{j'\in J'\subset I\setminus \{i\}} 
\rho_{j'}\overline{\tilde{h}^{U_{j'}}}(\xi_{U_{j'}},\eta_{U_{j'}}).
\end{align*}

The family $\{\langle \xi_{U_i}, \eta_{U_i}\rangle_{U_i}\}_{i\in I}$ defined as above satisfies the condition \eqref{hermicondi}; 
hence it is a $d$-Hermitian metric for $d$-complex bundle $E$ on the $d$-complex manifold $(X,\dholos{X})$.
\end{proof}

\section{Differential operators}\label{section3}
In this section, we first describe some basic concepts concerning differential operators in the context of 
$d$-complex vector bundles on $d$-complex manifolds. One of the main theorem of this section is Theorem 
\ref{ellipticdec}, which uses numerous results from Functional Analysis that involves working with Sobolev 
spaces (see Appendix \ref{appendix:sectionA}) and the theory of pseudo-differential operators 
(see \S \ref{subsection3.1}).

Let $E$ and $F$ be two $d$-complex bundles on a $d$-complex manifold $(X,\dholos{X})$ of dimension $n$ 
with $d$-holomorphic atlas $\mathcal{U}=\{(U_i,\phi_i)\}_{i\in I}$.

Recall that an element $P\in Hom_\mathcal{C}(E, F)$ (see \cite[Section 3.1]{SAAJ}) can be expressed as
a family $\{P_{U_i}:E(U_i)\rightarrow F(U_i)\}_{i\in I}$ of $\mathbb{C}$-linear morphisms with respect to the $d$-holomorphic atlas $\mathcal{U}$ satisfying
\begin{equation}\label{diffcompati}
P_{U_j}(s_{U_j}) = \left\{\begin{array}{lr}
        P_{U_i}(s_{U_i}), & \text{if } \trans{j}{i}\text{ is holomorphic};\vspace{7pt}\\
        \overline{P_{U_i}(s_{U_i})}, & \text{if } \trans{j}{i}\text{ is anti-holomorphic},
        \end{array}\right.
\end{equation}
for a global section $s=\{s_{U_i}\}_{i\in I}$ of the bundle $E$.

An element $P=\{P_{U_i}\}_{i\in I}\in Hom_\mathcal{C}(E,F)$ is called a differential operator of order 
$m$ if each $P_{U_i}$ ($i\in I$) is a differential operator of order $m$. For more details about the local 
description of differential operators of order $m$, see \cite[Chapter 4, Section 2]{ROW}.
Let $\mathrm{Diff}_m(E,F)$ denote the space of all $m$-th order differential operators between $d$-complex 
bundles $E$ and $F$.

\begin{proposition}
Let $P=\{P_{U_i}\}_{i\in I}\in \mathrm{Diff_m(E, F)}$ be a differential operator of order $m$. Then, 
there exists a family of linear maps 
$$
\{\sigma_m(P)_{U_i}:\big(T^\star_X\otimes \mathcal{L}\big)(U_i)\rightarrow Hom_{\mathcal{A}^0_d}(E|_{U_i},F|_{U_i})\}_{i\in I}
$$ 
with respect to the $d$-holomorphic atals $\mathcal{U}=\{(U_i,\phi_i)\}_{i\in I}$, satisfying the compatibility 
condition
\begin{equation}\label{symbol3}
\sigma_m(P)_{U_j}(\nu_{U_j})(s_{U_j}) = \left\{\begin{array}{lr}
       \sigma_m(P)_{U_i}(\nu_{U_i})(s_{U_i}), & \text{if } \trans{j}{i}\text{ is holomorphic};\vspace{7pt}\\
        \overline{\sigma_m(P)_{U_i}(\nu_{U_i})(s_{U_i})}, & \text{if } \trans{j}{i}\text{ is anti-holomorphic},
        \end{array}\right.
\end{equation}
for global sections $\nu=\{\nu_{U_i}\}_{i\in I}\in (T^\star X\otimes \mathcal{L})(X)$ and $s=\{s_{U_i}\}_{i\in I}\in E(X)$ with appropriate compatibility conditions. Each $\sigma_m(P)_{U_i}$ ($i\in I$) is 
uniquely defined such that for every open subset $V$ of $X$, $\nu_V\in (T^\star_X\otimes \mathcal{L})(V)$ (non-trivial) and $s_V\in E(V)$, we have
\begin{equation}\label{symbeq.}
\sigma_m(P)_V(\nu_V)(s_V)=P_V\left(\frac{i^m}{m!}\big(f_V-f_V(x)\big)^m s_V\right)\quad\text{ where }d_\mathcal{L}(f_V)=\nu_V.
\end{equation}
\end{proposition}
\begin{proof}
Let $f=\{f_{U_i}\}_{i\in I}\in \mathcal{L}(X)$ be a smooth section such that 
$d(f_{U_i})=\nu_{U_i}$ and $s=\{s_{U_i}\}_{i\in I}\in E(X)$. Define $\sigma(P)_{U_i}(\nu_{U_i})(s_{U_i})(x)=P_{U_i}\left(\frac{i^m}{m!}(f_{U_i}-f_{U_i}(x))s_{U_i}\right)(x)$ $(i\in I)$, which is uniquely defined independent of choices made of $f=\{f_{U_i}\}_{i\in I}$ and $s=\{s_{U_i}\}_{i\in I}$ \cite[Chapter 4, Section 2, Page no. 115]{ROW}.

\noindent \textbf{Case-1:}
If transition map $\trans{j}{i}:\phi_i(U_i\cap U_j)\rightarrow \phi_j(U_i\cap U_j)$ is holomorphic, then 
\begin{align*}
\sigma_m(P)_{U_i}(\nu_{U_i})(s_{U_i})(x)&=P_{U_i}\left(\frac{i^m}{m!}(f_{U_i}-f_{U_i}(x))^ms_{U_i}\right)(x)\\
&=P_{U_j}\left(\frac{i^m}{m!}(f_{U_j}-f_{U_j}(x))^ms_{U_j}\right)(x)\quad(\text{from equation }\eqref{diffcompati})\\
&=\sigma_m(P)_{U_j}(\nu_{U_j})(s_{U_j})(x)\quad(\text{for }x\in U_i\cap U_j)
\end{align*}
\noindent \textbf{Case-2:}
If transition map $\trans{j}{i}:\phi^{-1}_i(U_i\cap U_j)\rightarrow \phi_j(U_i\cap U_j)$ is anti-holomorphic, then
\begin{align*}
\sigma_m(P)_{U_i}(\nu_{U_i})(s_{U_i})(x)&=P_{U_i}\left(\frac{i^m}{m!}(f_{U_i}-f_{U_i}(x))^ms_{U_i}\right)(x)\\
&=\overline{P_{U_j}\left(\frac{i^m}{m!}(f_{U_j}-f_{U_j}(x))^ms_{U_j}\right)}(x)\quad(\text{from equation }\eqref{diffcompati})\\
&=\overline{\sigma_m(P)_{U_j}(\nu_{U_j})(s_{U_j})}(x)\quad(\text{for }x\in U_i\cap U_j)
\end{align*}
From the above discussion, for a given $P=\{P_{U_i}\}_{i\in I}\in \mathrm{Diff_m(E, F)}$, we have a uniquely defined
family $\{\sigma_m(P)_{U_i}\}_{i\in I}$ satisfying the condition \eqref{symbol3}.

Let $V$ be an open subset of $X$, $s_V=\{s_{U_i\cap V}\}_{i\in I}\in E(V)$ and 
$f_V=\{f_{U_i\cap V}\}\in \mathcal{L}(V)$ such that 
$d(f_{U_i\cap V})=\nu_{U_i\cap V}\in \big(T^\star_X\otimes \mathcal{L}\big)(V)$ then,
\begin{align*}
\sigma_k(P)_V(\nu_V)(s_V)&=\{\sigma_k(P)_{V\cap U_i}(\nu_{V\cap U_i})(s_{V\cap U_i})\}_{i\in I}\\
&=\left\{P_{V\cap U_i}\left(\frac{i^m}{m!}\big(f_{V\cap U_i}-f_{V\cap U_i}(x)\big)^ms_{V\cap U_i}\right)\right\}_{i\in I}\\
&=P_V\left(\frac{i^m}{m!}\big(f_V-f_V(x)\big)^ms_V\right)\quad \big(\in F(V)\big).
\end{align*}
\end{proof}
\subsection{Pseudo-Differential operators}\label{subsection3.1}
Let $(X,\dholos{X})$ be a $d$-complex manifold of real dimension $2n$ with a $d$-holomorphic atlas 
$\mathcal{U}=\{(U_i,\phi_i)\}_{i\in I}$, which is also a smooth atlas on the associated smooth manifold and $\big(U\rightarrow\mathcal{D}(U)\big)$ be the sheaf 
of smooth functions with compact support on $X$. 
This section is based on \cite[Chapter 4, Section 3]{ROW} and \S \ref{appendix2}.

\begin{definition}\rm{
An element $L=\{L_{U_i}\}_{i\in I}\in 
\mathrm{Hom}_\mathbb{R}\big(\mathcal{D}(X),\mathcal{A}^0(\mathcal{L})(X)\big)$ is 
called pseudo-differential operator on $X$ if for each $i\in I$ and any open subset 
$U_i'\subset \overline{U}_i'\subset U_i$, there exist a polynomial $p_i\in S^m_0(\phi_i(U_i))$ 
(see \ref{appendix2}) such that, after extending each $f\in \mathcal{D}(U_i')$ by zero 
to an element of $\mathcal{D}(U_i)$ 
\begin{align*}
L_{U_i}(f)(x)&=L_{U_i}(p_i)(f)(x)=L_{U_i}(p_i)(f\circ \phi_i^{-1})(\phi_i(x))\\
&=\int_{\mathbb{R}^{2n}}p_i(\phi_i(x),y)\widehat{f\circ \phi_i^{-1}}(y) e^{i\langle \phi_i(x),y \rangle} dy\quad(\text{see }\ref{appendix2})
\end{align*}
where $\widehat{f\circ \phi_i^{-1}}(y)=(2\pi)^{-n}\int_{\mathbb{R}^{2n}}e^{-i\langle x,y \rangle} f\circ \phi_i^{-1}(x) dx$ is the fourier transform of $f\circ \phi_i^{-1}$.
}
\end{definition}

Let $E$ and $F$ be smooth $d$-complex vector bundles on a $d$-complex manifold $(X,\dholos{X})$ of ranks 
$l$ and $m$ respectively with same trivializing cover as the atlas $\mathcal{U}$ has. An element $L=\{L_{U_i}\}_{i\in I}\in 
Hom_\mathcal{C}\big(\mathcal{D}(E),\mathcal{A}^0(F)\big)$ is called a pseudo-differential 
operator between $E$ and $F$ if for each $i\in I$ and any open subset ${U}_i'\subset\overline{U}_i'\subset U_i$, 
there exist a matrix $p=(p_{rs})\in \mathrm{Mat}_{m\times l}(S^m_0(U_i))$ (see \ref{appendix2}) such that after 
extending each element $s_{U_i}\in\mathcal{D}(E)(U_i')\cong [\mathcal{D}(\mathcal{C})(U_i')]^l$ 
by zero to an element of $\mathcal{D}(E)(U_i)\cong [\mathcal{D}(\mathcal{C})(U_i)]^l$
$$
L_{U_i}(s_{U_i})(x)=L_{U_i}(p)(s_{U_i}\circ \phi_i^{-1})(\phi_i(x)).
$$
The operator $L_{U_i}(p)=[L_{U_i}(p_{rs})]$, is a matrix of canonical pseudo-differential operators 
$L_{U_i}(p_{rs})$ $(1\leq r \leq m; 1\leq s \leq l)$.

\begin{remark}\rm{
The local $m$-symbol of a pseudo-differential operator $L\in \mathrm{PDiff}(E,F)$ with respect to a 
fixed trivializations of $E$ and $F$ on some co-ordinate chart $(U_i,\phi_i)$, is the matrix
$$
\sigma_m(L)_{U_i}(x,y)=\sigma_m(p)(x,y)=(\sigma_m(p_{rs}))(x,y)\hspace{5pt}(1\leq r \leq m; 1\leq s \leq l),
$$
where $\sigma_m(p_{rs})(x,y)=\displaystyle\lim_{\lambda\to \infty} \frac{p_{rs}(x,\lambda y)}{\lambda^m}$ and $(x,y)\in U_i\times \mathbb{R}^{2n}\simeq\big(T^\star_X\otimes \mathcal{L}\big)(U_i)$ ($y\neq 0$).

The order of the local symbol may depends on the co-ordinate chart $(U_i,\phi_i)$, as we require 
$p_{rs}\in S^m_0(U_i)$ ($1\leq r\leq m;1\leq s\leq l$) without requiring the same value $``m"$ for each 
chart $(U_i,\phi_i)$ $(i\in I)$.
}
\end{remark}

\begin{theorem}\label{adjoint3}
Let $(U_i,\phi_i)$ and $(U_j,\phi_j)$ ($U_i\cap U_j\neq \phi$) be two relatively compact co-ordinate 
charts on the $d$-complex manifold $(X,\dholos{X})$ i.e. $\phi_i(U_i),\phi_j(U_j)\subset \mathbb{R}^{2n}$ 
have compact closures. The transition map 
$\trans{i}{j}=\phi_i\circ \phi_j^{-1}:\phi_j(U_j)\rightarrow \phi_i(U_i)$, which is either holomorphic or 
anti-holomorphic, is also a diffeomorphism from $\phi_j(U_j)$ onto $\phi_i(U_i)$. Let 
$L=\{L_{U_i}\}_{i\in I}\in \mathrm{Hom}_\mathcal{C}\big(\mathcal{D}(\mathcal{C})(X),\mathcal{C}(X)\big)$ 
and $p_i\in S^m_0(\phi_i(U_i))$ be such that $L_{U_i}\equiv L_{U_i}(p_i)$. Then, there is a polynomial 
$q_j\in S^m_0(\phi_j(U_j))$ such that $L_{U_j}\equiv L_{U_j}(q_j)$ and 
\[
\sigma_m(q_j)(\phi_j(x),\eta) = \left\{\begin{array}{lr}
\sigma_m(p_i)\left(\phi_i(x),\left[D(\trans{i}{j})^t\right]^{-1}\eta\right), 
& \text{if } \trans{j}{i} \text{ is holomorphic};\vspace{7pt}\\
\overline{\sigma_m(p_i)}\left(\phi_i(x),-\left[D(\trans{i}{j})^t\right]^{-1}\eta\right), 
& \text{if } \trans{j}{i} \text{ is anti-holomorphic},
\end{array}\right.
\]
for $(\phi_j(x),\eta)\in \phi_j(U_i\cap U_j)\times \mathbb{R}^{2n}\simeq \big(T^\star_X\otimes \mathcal{L}\big)(U_i\cap U_j)$ $(\eta\neq 0)$.
\end{theorem}
\begin{proof}
Let $s=\{s_{U_i}\}_{i\in I}\in \mathcal{C}(X)$ with respect to the atlas $\mathcal{U}$ such that each $s_{U_i}$ $(i\in I)$ has compact support i.e. $s_{U_i}\in \mathcal{D}(\mathcal{C})(U_i)$, we break the proof in two cases as follows.

\noindent \textbf{Case-1:}
If $\phi_j\circ \phi_i^{-1}$ is holomorphic. Then $L_{U_j}(s_{U_j})(x)=L_{U_i}(s_{U_i})(x)=L_{U_i}(p)(s_{U_i}\circ \phi_i^{-1})(\phi_i(x))$ and from \cite[Theorem 3.9]{ROW}, we have a polynomial $q(x,y)\in S^m_0(\phi_j(U_i\cap U_j))$ such that 
\begin{align*}
L_{U_j}(s_{U_j})(x)&=L_{U_j}(q)(s_{U_j}\circ\phi_j^{-1})(\phi_j(x))\quad\text{ and }\\
\sigma_m(q)(\phi_j(x),\eta)&=\sigma_m(p)\left(\phi_i(x),\left[D(\trans{i}{j})^t\right]^{-1}\eta\right).
\end{align*}
\noindent \textbf{Case-2:}
If $\phi_j\circ \phi_i^{-1}$ is anti-holomorphic, we have
\begin{align*}
L_{U_i}(s_{U_i})(x)&=L_{U_i}(p)(s_{U_i})(x)=\displaystyle\int_{\mathbb{R}^{2n}}p(\phi_i(x),y).\hspace{5pt} \widehat{s_{U_i}\circ\phi_i^{-1}}(y)\hspace{5pt}e^{i\langle \phi_i(x),y \rangle} dy\quad \text{and }\\
L_{U_j}(s_{U_j})=\overline{L_{U_i}(s_{U_i})}(x)&=\displaystyle\int_{\mathbb{R}^{2n}}\overline{p}(\phi_i(x),y).\hspace{5pt} \widehat{\overline{s_{U_i}\circ\phi_i^{-1}}}(-y)\hspace{5pt}e^{i\langle \phi_i(x),-y \rangle} dy\quad\text{ (property of fourier transform)}\\
&=\displaystyle\int_{\mathbb{R}^{2n}}\overline{p}(\phi_i(x),-y).\hspace{5pt} \widehat{\overline{s_{U_i}\circ\phi_i^{-1}}}(y)\hspace{5pt}e^{i\langle \phi_i(x),y \rangle} dy
\end{align*}

From \cite[Theorem 3.9]{ROW}, we have a polynomial $q\in S^m_0(\phi_j(U_i\cap U_j))$ such that 
\begin{align*}
L_{U_j}(s_{U_j})(x)=L_{U_j}(q_j)(s_{U_j})(x)&=\displaystyle\int_{\mathbb{R}^{2n}}q_j(\phi_j(x),\eta).\hspace{5pt} \widehat{s_{U_j}\circ\phi_j^{-1}}(\eta)\hspace{5pt}e^{i\langle \phi_j(x),\eta \rangle} d\eta\quad\text{ and }\\
\sigma_m(q_j)(\phi_j(x),\eta)&=\overline{\sigma_m(p_i)}\left(\phi_i(x),-\left[D(\trans{i}{j})^t\right]^{-1}\eta\right).
\end{align*}
where $(\phi_j(x),\eta)\in \phi_j(U_j\cap U_i)\times \mathbb{R}^{2n}$ and $(\phi_i(x),y)\in \phi_i(U_i\cap U_j)\times \mathbb{R}^{2n}$ satisfying $y=-\left[D(\trans{i}{j})^t\right]^{-1}\eta$.
\end{proof}

Now, we can define the order of pseudo-differential operators as follows.

\begin{definition}\rm{
A pseudo-differential operator $L=\{L_{U_i}\}_{i\in I}\in \mathrm{PDiff}(E,F)$ with respect to the
$d$-holomorphic atlas $\mathcal{U}=\{(U_i,\phi_i)\}_{i\in I}$ of the $d$-complex manifold $X$, is said 
to be of order $m$, if for any choice of co-ordinate chart $U_i\subset X$, the operator $L_{U_i}$ is of 
order $m$.
}
\end{definition}

Let $\mathrm{PDiff}_m(X)$ denote the collection of all pseudo-differential operators on the $d$-complex 
manifold $(X,\dholos{X})$ of order $m$.

\begin{proposition}\label{adjoint2}
Let $(X,\dholos{X})$ be a compact $d$-complex manifold. Then $\mathrm{PDiff}_m(X)\subset \mathrm{OP}_m(X)$.
\end{proposition}
\begin{proof}
The result follows as in \cite[Chapter 4, Propostion 3.13]{ROW}, using properties of the operator $L(p)$ 
(see Definition \ref{pseuddef} \S \ref{appendix2}) and definition of Sobolev norm on compact $d$-complex manifold 
(see \S \ref{appendix1}).
\end{proof}

\begin{definition}\rm{
Let $E$ and $F$ be two $d$-complex vector bundles on a $d$-complex manifold $(X,\dholos{X})$ and $T'_X\otimes \mathcal{L}\xrightarrow{\pi}X$ be the co-tangent bundle taking values in orientation lie bundle $\mathcal{L}$ with trivial section deleted, define
$$
\mathrm{Smbl}_m(E,F)=\{\sigma_m:\big(T'_X\otimes \mathcal{L}\big)(X)\rightarrow Hom_\mathcal{C}(E,F)\mid \sigma_m(x,\rho\nu)=\rho^m\sigma(x,\nu)\text{ for }\rho>0\}
$$
For $\nu=\{\nu_i\}_{i\in I}\in (T'_X\otimes \mathcal{L})(X)$ and $s=\{s_{U_i}\}_{i\in I}\in E(X)$, an element $\sigma\in \mathrm{Smbl}_m(E,F)$ can be expressed as a family $\{\sigma_{U_i}\}_{i\in I}$ with respect to the atlas $\mathcal{U}=\{(U_i,\phi_i\}_{i\in I}$ satisfying the compatibility condition
\[
    \sigma_{U_j}(\nu_j)(s_{U_j}) = \left\{\begin{array}{lr}
        \sigma_{U_i}(\nu_i)(s_{U_i}), & \text{if } \trans{j}{i} \text{ is holomorphic};\\
        \overline{\sigma_{U_i}(\nu_i)(s_{U_i})}, & \text{if } \trans{j}{i} \text{ is anti-holomorphic}
        \end{array}\right.
  \]
}
\end{definition}

For simplicity denote the collection $\mathrm{Smbl}_m(\mathcal{C},\mathcal{C})$ by $\mathrm{Smbl}_m(X)$, where $\mathcal{C}$ is the trivial $d$-complex line bundle on the $d$-complex manifold $(X,\dholos{X})$.

\begin{proposition}\label{sigmaprop}
For a given $d$-complex manifold $(X,\dholos{X})$, there exists a canonical linear map
$$
\sigma_m:\mathrm{PDiff}_m(X)\rightarrow \mathrm{Smbl}_m(X).
$$
Which is defined locally in a co-ordinate chart with $x\in U\subset X$ by
$$
\sigma_m(L_U)(x,\xi)=\sigma_m(p)(x,\xi),
$$
where $L_U=L(p)\in \mathrm{PDiff}_m(X)(U)$ and $(x,\xi)\in U\times \mathbb{R}^{2n}\simeq \big(T'_X\otimes \mathcal{L}\big)(U) $.
\end{proposition}
\begin{proof}
The proof follows using the Theorem \ref{adjoint3}.
\end{proof}
The above procedure generalizes to pseudodifferential operators mapping sections of $d$-complex vector bundle $E$ to sections of $d$-complex vector bundle $F$. We shall denote the space of pseudodifferential operatos of order $m$ mapping $\mathcal{D}(E)(X)$ into $\mathcal{A}^0(F)(X)$ by $\mathrm{PDiff_m(E, F)}$. We can have an analog to Proposition \ref{sigmaprop} for $d$-complex bundles.
\begin{proposition}\label{sigmaprop1}
Let $E$ and $F$ be $d$-complex vector bundle over a $d$-complex manifold $(X,\dholos{X})$. There exists a canonical linear map
$$
\sigma_m:\mathrm{PDiff_m(E, F)}\rightarrow \mathrm{Smbl}_m(E,F).
$$
The map $\sigma_m$, is defined locally in a co-ordinate chart with $x\in U\subset X$ by
$$
\sigma_m(L_U)(x,\xi)=[\sigma_m(p_{ij})(x,\xi)]
$$
where $L_U=[L(p_{ij})]\in \mathrm{PDiff}(E_{U},F_{U})$ is the matrix of canonical pseudo-differential operators on $X$ and $(x,\xi)$ is an element of $U\times \mathbb{R}^{2n}\simeq \big(T'_X \otimes \mathcal{L}\big)(U)$.
\end{proposition}

\begin{theorem}\label{symbl}
Let $E$ and $F$ be $d$-complex vector bundles on $d$-complex manifold $(X,\dholos{X})$. Then the following sequence is exact
$$
0\rightarrow K_m(E,F)\rightarrow \mathrm{PDiff_m(E, F)}\xrightarrow{\sigma_m}\mathrm{Smbl}_m(E,F)\rightarrow 0
$$
where $\sigma_m$ is canonical symbol map as described in the Proposition \ref{sigmaprop1}, $K_m(E,F)$ is kernel of $\sigma_m$ and $K_m(E,F)\subset \mathrm{OP}_{m-1}(E,F)$, if $X$ is compact.
\end{theorem}
\begin{proof}
The proof follows as in \cite[Chapter 4, Section 3, Theorem 3.16]{ROW} using smooth partition of unity.
\end{proof}

\subsection*{Adjoint operator}\label{subsecad}
We henceforth consider d-complex manifold to be compact, although many of the results we will discuss also 
hold for non-compact $d$-complex manifolds. Let $E$ be a smooth $d$-Hermitian vector bundle ($d$-complex 
vector bundle with $d$-Hermitian metric) on a $d$-Hermitian manifold $(X,\dholos{X})$ ($d$-complex manifold 
with $d$-Hermitian metric) with $d$-holomorphic atals $\mathcal{U}=\{(U_i,\phi_i)\}_{i\in I}$. 

Recall the constructio of line bundle $\mathcal{L}$ as in Subsection \ref{subsec2.1} with same trivializing cover as the 
$d$-holomorphic atlas $\mathcal{U}$ on $X$, with the cocycle map
\[
    \cocycle{j}{i}{\mathcal{L}}(x) = \left\{\begin{array}{lr}
        1, & \text{if transition map } \trans{j}{i} \text{ is holomorphic};\vspace{7pt}\\
        -1, & \text{if transition map } \trans{j}{i} \text{ is anti-holomorphic}.
        \end{array}\right.\quad(\text{for }x\in U_i\cap U_j)
  \] 

Given a Riemannian metric on the smooth manifold associated to  $(X,\dholos{X})$, we have a smooth 
measure (respectively, smooth measure taking value in line bundle $\mathcal{L}$), denoted by 
$\star 1$ (respectively, $\star 1_\mathcal{L}$).

Let $\mathcal{A}^0$ $(\text{respectively, }\mathcal{A}^0(\mathcal{L}))$ be the sheaf of smooth functions 
(respectively, sheaf of smooth functions taking values in $\mathcal{L}$) on the smooth manifold associated 
to $(X,\dholos{X})$. Let $\mathcal{A}^0_\mathcal{C}(E)$ be the sheaf of smooth sections of smooth $d$-complex 
bundle $E$ on the $d$-complex manifold $(X,\dholos{X})$. For $\alpha,\beta\in \mathcal{A}^0_\mathcal{C}(E)(X)$, 
we have $\langle\alpha,\beta\rangle_E$, a $d$-complex valued smooth function on $(X,\dholos{X})$, which we 
can write as $\langle\alpha,\beta\rangle_{E,x}=A_{\alpha\beta,1}(x)+iA_{\alpha\beta,2}(x)$, where 
$A_{\alpha\beta,1}(x)\in \mathcal{A}^0(X)$ and $A_{\alpha\beta,2}(x)\in \mathcal{A}^0(\mathcal{L})(X)$. 
Using the measures $\star 1_\mathcal{L}$ and $\star 1$, we define an inner product $(\, , \,)$ on 
$\mathcal{A}^0_\mathcal{C}(E)(X)$ such that for $\alpha,\beta\in \mathcal{A}^0(X)$
$$
(\alpha,\beta)=\int_X A_{\alpha\beta,1} (\star 1)+\int_X A_{\alpha\beta,2} (\star 1_\mathcal{L}).
$$
Let $||\alpha||_{L^2}=(\alpha,\alpha)^{\frac{1}{2}}$ be the respective $L^2$-norm of 
$\alpha\in \mathcal{A}_\mathcal{C}^0(E)(X)$ and $W^0_\mathcal{C}(E)(X)$ be the completion of 
$\mathcal{A}_\mathcal{C}^0(E)(X)$ with respect to this $L^2$-norm. The space $W^0_\mathcal{C}(E)(X)$ 
contains $\mathcal{A}_\mathcal{C}^0(E)(X)$ along with some more global sections. Although, we have a 
sequence of smaller Hilbert spaces $\{W^s_\mathcal{C}(E)\}_{s\in \mathbb{N}}$ containing 
$\mathcal{A}_\mathcal{C}^0(E)(X)$ (see appendix \ref{appendix1}).

\begin{definition}\rm{
Let $E,F$ be $d$-complex vector bundles with $d$-Hermitian metric on the $d$-complex manifold 
$(X,\dholos{X})$ and $L\in Hom_\mathcal{C}(E,F)$, $S\in Hom_\mathcal{C}(F,E)$. The element $S$ 
is called adjoint of $L$ if 
$$
\langle Lf,g \rangle=\langle f,Sg \rangle
$$
for all $f\in \mathcal{A}^0_\mathcal{C}(E)$ and $g\in \mathcal{A}^0_\mathcal{C}(F)$.
}
\end{definition}

\begin{remark}\rm{
If the adjoint of an element $L\in Hom_\mathcal{C}(E, F)$ exists, then it is unique, usually denoted by $L^\star$.
}
\end{remark}

\begin{proposition}\label{adjoint}
If $L\in \mathrm{Diff}_m(E,F)$, then there exist an adjoint $L^\star$ and moreover, 
$L^\star\in \mathrm{Diff}_m(F,E)$.
\end{proposition}
\begin{proof}
Let $L=\left\lbrace L_{U_i}\right\rbrace_{i\in I}\in \mathrm{Diff_m(E, F)}$ with respect to some atlas 
$\mathcal{U}=\{(U_i,\phi_i)\}_{i\in I}$, be differential operator of order $k$. Using the inner product 
on $\mathcal{A}^0_\mathcal{C}(F)(U_i)$, we can describe $L_{U_i}^\star$ uniquely, for each $i\in I$
(see \cite[Chapter 4, Proposition 2.8]{ROW}).

For $\xi\in \mathcal{A}_\mathcal{C}^0(E)(X)$ and $\eta \in \mathcal{A}_\mathcal{C}^0(F)(X)$, we have
\[
    \langle L_{U_j}(\xi_{U_j}),\eta_{U_j}\rangle = \left\{\begin{array}{lr}
        \langle L_{U_i}(\xi_{U_i}),\eta_{U_i}\rangle, & \text{if } \trans{j}{i} \text{ is holomorphic};\vspace{7pt}\\
        \overline{\langle L_{U_i}(\xi_{U_i}),\eta_{U_i}\rangle}, & \text{if } \trans{j}{i} \text{ is anti-holomorphic},
        \end{array}\right.
  \]
and by uniqueness of formal adjoint operator $L_{U_i}^\star$, we have
\begin{equation}\label{adjeq}
L^\star_{U_j}(\eta_{U_j}) = \left\{\begin{array}{lr}
L^\star_{U_i}(\eta_{U_i}), & \text{if } \trans{j}{i} \text{ is holomorphic};\vspace{7pt}\\
\overline{L^\star_{U_i}(\eta_{U_i})}=(\overline{L}_{U_i})^\star(\overline{\eta}_{U_i}), & 
\text{if } \trans{j}{i} \text{ is anti-holomorphic}.
        \end{array}\right.
\end{equation}
Thus, we get a family $\{L_{U_i}^\star\}_{i\in I}$ with respect to the atlas 
$\mathcal{U}=\{(U_i,\phi_i\}_{i\in I}$, satisfying the condition \eqref{adjeq}. This proves that 
$L^\star = \{L_{U_i}^\star\}_{i\in I} \in \mathrm{Diff}_m(F,E)$.
\end{proof}

\begin{proposition}\label{adjoint1}
If $L\in \mathrm{OP}_m(E,F)$ {\rm (see Definition \ref{diffOP})}, then there exist an adjoint 
$L^\star\in \mathrm{Hom}_\mathcal{C}(F,E)$ and moreover, $L^\star\in \mathrm{OP}_m(F,E)$ with the 
extension
$$
(L^\star)_s:W^{s-m}_\mathcal{C}(F)(X)\rightarrow W^s_\mathcal{C}(E)(X)
$$
given by the adjoint map,
$$
(L_{s-m})^\star:W^{s-m}_\mathcal{C}(F)(X)\rightarrow W^s_\mathcal{C}(E)(X).
$$
\end{proposition}
\begin{proof}
The candidate $(L_{s-m})^\star:W^{s-m}_\mathcal{C}(F)(X)\rightarrow W^s_\mathcal{C}(E)(X)$ 
(for each $s\in \mathbb{Z}$) exists using the Propositon \ref{adjoint} and we have the proof of the 
theorem following \cite[Chapter 4, Section 2, Proposition 2.1]{ROW}, using the uniqueness of adjoint
and the Proposition \ref{Sobolev}.
\end{proof}

\begin{theorem}\label{symblprop}
Let $E,F$ and $G$ be $d$-complex vector bundles over a $d$-complex manifold $(X,\dholos{X})$. Then
\begin{enumerate}
\item If $Q\in \mathrm{PDiff}_r(E,F)$ and $P\in \mathrm{PDiff}_s(F,G)$, we have the composition $P\circ Q\in \mathrm{PDiff}_{r+s}(E,G)$ as operators, moreover
$$
\sigma_{r+s}(P\circ Q)=\sigma_s(P).\sigma_r(Q).
$$
\item If $P\in \mathrm{PDiff_m(E, F)}$, then $P^\star$, the adjoint of $P$, exists and $P^\star\in \mathrm{PDiff}_m(F,E)$, moreover
$$
\sigma_m(P^\star)=\sigma_m(P)^\star.
$$
\end{enumerate}
\end{theorem}
\begin{proof}
The proof follows as in \cite[Theorem 3.17]{ROW} using the adjoint property of differential operators, and the Theorems \ref{adjoint1}, \ref{adjoint3} and \ref{adjoint2}. 
\end{proof}

\subsection{Elliptic complexes}
We can indeed develop the theory of parametrixes for elliptic differential operators on $d$-complex bundles 
over $d$-complex manifolds. This can be achieved by utilizing results from functional analysis, 
Theorem \ref{symbl}, \ref{symblprop}, and following a similar line of argument as presented 
in \cite[Chapter 4, \S 4]{ROW}, detailed in Appendix \ref{appendix3}.

The fundamental decomposition theorem for elliptic complexes (refer to Theorem \ref{ellipticdec}) is crucial 
for describing both the Hodge decomposition for the Dolbeault complex associated with the $d$-complex manifold 
$(X,\dholos{X})$ and the de Rham complex for the smooth manifold $X$.

\begin{definition}\rm{
For a given sequence of $d$-complex vector bundles $(E_i)_{(0\leq i\leq N)}$ and differential operators $L_i:\mathcal{A}^0_\mathcal{C}(E_i)(X)\rightarrow \mathcal{A}^0_\mathcal{C}(E_{j+1})(X)(0\leq i\leq N-1)$, the sequence
$$
\mathcal{A}^0_\mathcal{C}(E_0)(X)\xrightarrow{L_0}\mathcal{A}^0_\mathcal{C}(E_1)(X)\xrightarrow{L_1}\mathcal{A}^0_\mathcal{C}(E_2)(X)\xrightarrow{L_2}\dots\xrightarrow{L_{N-1}}\mathcal{A}^0_\mathcal{C}(E_N)(X)
$$  
is called a complex iff $L_j\circ L_{j-1}\equiv 0$ $(j=1,\dots, N-1)$. The complex is denoted by $E_\bullet$.
}
\end{definition}
\begin{definition}\rm{
The complex 
$$
E_\bullet:\mathcal{A}^0_\mathcal{C}(E_0)(X)\xrightarrow{L_0}\mathcal{A}^0_\mathcal{C}(E_1)(X)\xrightarrow{L_1}\mathcal{A}^0_\mathcal{C}(E_2)(X)\xrightarrow{L_2}\dots\xrightarrow{L_{N-1}}\mathcal{A}^0_\mathcal{C}(E_N)(X)
$$ 
is called an elliptic complex if the following associated symbol sequence is exact
$$
0\rightarrow \pi^\star E_0\xrightarrow{\sigma(L_0)}\pi^\star E_1\xrightarrow{\sigma(L_1)}\dots\xrightarrow{\sigma(L_{N-1})}\pi^\star E_N\rightarrow 0.
$$ 
}
\end{definition}
\begin{remark}\rm{
Considering the de-Rham complex
$$
0\rightarrow\mathcal{A}^0_\mathcal{C}(X)\xrightarrow{d_\mathcal{C}}\mathcal{A}^1_\mathcal{C}(X)\xrightarrow{d_\mathcal{C}}\mathcal{A}^2_\mathcal{C}(X)\xrightarrow{d_\mathcal{C}}\dots
$$
We have the associated 1-symbol sequence
$$
0\rightarrow\wedge^0{T^\star}^\mathcal{C}_X(X)\xrightarrow{\;\sigma_1(d_\mathcal{C})(\nu)\;}\wedge^1{T^\star}^\mathcal{C}_X(X)\xrightarrow{\;\sigma_1(d_\mathcal{C})(\nu)\;}\wedge^2{T^\star}^\mathcal{C}_X(X)\xrightarrow{\;\sigma_1(d_\mathcal{C})(\nu)\;}\dots.
$$
where $\nu\in \big({T^\star}^\mathcal{C}_X\otimes \mathcal{L}\big)(X)$.

Let $f=\{f_{U_i}\}_{i\in I}\in \big(\mathcal{C}\otimes \mathcal{L}\big)(X)$ be a smooth section such that 
$d(f_{U_i})=\nu_{U_i}$, and $s=\{s_{U_i}\}_{i\in I}\in \wedge^k{T^\star}^\mathcal{C}_X(X)$. Then, we have
\begin{align*}
\sigma_1(d_\mathcal{C})_{U_i}(\nu_{U_i})(s_{U_i})(x)&=d\Big(\sqrt{-1}\big(f_{U_i}-f_{U_i}(x)\big)s_{U_i}\Big)(x)\\
&=\sqrt{-1}d\big(f_{U_i}-f_{U_i}(x)\big)(x)\wedge s_{U_i}(x)+\sqrt{-1}\big(f_{U_i}-f_{U_i}(x)\big)(x)(d(s_{U_i})(x))\\
&=\sqrt{-1}\nu_{U_i}(x)\wedge s_{U_i}(x)
\end{align*} 
and 
\begin{align*}
\Big(\sigma_1(d_\mathcal{C})_{U_i}(\nu_{U_i})\circ \sigma_1(d_\mathcal{C})_{U_i}(\nu_{U_i})\Big)s_{U_i}(x)&=\sigma_1(d_\mathcal{C})_{U_i}(\nu_{U_i})\sqrt{-1}\big(d(f_{U_i})\wedge s_{U_i}\big)(x)\\
&=\sqrt{-1}\nu_{U_i}(x)\wedge i\nu_{U_i}(x)\wedge s_{U_i}(x)\\
&=0
\end{align*}
Similarly, for the Dolbeault complex
$$
0\rightarrow \mathcal{A}^{(p,0)}(X)\xrightarrow{\overline{\partial}}\mathcal{A}^{(p,1)}(X)\xrightarrow{\overline{\partial}}\mathcal{A}^{(p,2)}(X)\rightarrow\dots\quad(\text{for }0\leq p\leq n),
$$

we have the associated $1$-symbol sequence
$$
0\rightarrow\wedge^0{T^\star}^{(p,1)}(X)\xrightarrow{\;\sigma_1(\overline{\partial})(\nu)\;}\wedge^1{T^\star}^{(p,1)}(X)\xrightarrow{\;\sigma_1(\overline{\partial})(\nu)\;}\wedge^2{T^\star}^{(p,1)}(X)
\xrightarrow{\;\sigma_1(\overline{\partial})(\nu)\;}\dots.
$$
Here,
$$
\sigma_1(\overline{\partial})(\nu_{U_i})(s_{U_i})(x)=\sqrt{-1}\nu_{U_i}^{(0,1)}(x)\wedge (s_{U_i})(x)
$$
for $\nu=\{\nu_{U_i}\}_{i\in I}=\{d(f_{U_i})\}_{i\in I}\in ({T^\star}^\mathcal{C}_X\otimes \mathcal{L})(X)$ and $s=\{s_{U_i}\}_{i\in I}\in \wedge^kT^{\star(p,1)}(X))$. Also, 
\begin{align*}
\Big(\sigma_1(\overline{\partial})(\nu_{U_i})\circ \sigma_1(\overline{\partial})(\nu_{U_i})\Big)s_{U_i}(x)
&= \sigma_1(\overline{\partial})(\nu_{U_i})\sqrt{-1}\big(d(f_{U_i})^{(0,1)}\wedge s_{U_i}\big)(x)\\
&=\sqrt{-1}\nu_{U_i}^{(0,1)}(x)\wedge i\nu_{U_i}^{(0,1)}(x)\wedge s_{U_i}(x)\\
&=0
\end{align*}
Hence, the de-Rham complex \eqref{derhamcomp} as well as Dolbeault complex \eqref{dolbeaultcomp} are elliptic complex.}
\end{remark}
\begin{definition}\rm{
For a given elliptic complex $E_\bullet$, define the cohomology group,
$$
H^q(E_\bullet)=\frac{\text{Ker}\big(L_q:\mathcal{A}^0_\mathcal{C}(E_q)(X)\rightarrow \mathcal{A}^0_\mathcal{C}(E_{q+1})(X)\big)}{\text{Image}\big(L_{q-1}:\mathcal{A}^0_\mathcal{C}(E_{q-1})(X)\rightarrow \mathcal{A}^0_\mathcal{C}(E_{q})(X)\big)}\quad q=0,1,\dots,N.
$$
Where, $L_{-1}=L_N=0=E_{-1}=E_{N+1}$
}
\end{definition}
Using Theorem \ref{adjoint}, each $L_j\in \mathrm{Diff}_1(E_j,E_{j+1})$ has an adjoint $L_j^\star\in \mathrm{Diff}_{-1}(E_{j+1},E_j)$.

Define the Laplacian operator $\Delta_j=L_j^\star\circ L_j+L_{j-1}\circ L_{j-1}^\star:\mathcal{A}^0_\mathcal{C}(X,E_j)\rightarrow \mathcal{A}^0_\mathcal{C}(X,E_j)$, which is self-adjoint and elliptic of order $2k$ (see \cite[Chapter 4, Section 5]{ROW}).

Denote by, $\mathcal{H}(\mathcal{A}^0_\mathcal{C}(E_j))=\mathcal{H}_{\Delta_j}(E_j)=Ker(\Delta_j)\subset \mathcal{A}^0_\mathcal{C}(X,E_j)$, the collection of $\Delta_j$-harmonic sections. By Theorem \ref{orthodeco}, we have a map
$$
G_j=G_{\Delta_j}:\mathcal{A}^0_\mathcal{C}(E_j)(X)\rightarrow \mathcal{A}^0_\mathcal{C}(E_j)(X),
$$
and the projection map
$$
H_j:\mathcal{A}^0_\mathcal{C}(E_j)(X)\rightarrow \mathcal{H}(\mathcal{A}^0_\mathcal{C}(E_j)(X))\subset \mathcal{A}^0_\mathcal{C}(E_j)(X).
$$

Using Theorem \ref{orthodeco} and the arguments as in \cite[Theorem 5.2]{ROW}, we have
\begin{theorem}\label{ellipticdec}\cite[Theorem 5.2]{ROW}
For an elliptic complex $(E_j,L_j)_{j\in I}$, we have
\begin{enumerate}
\item an orthogonal decomposition,
$$
\mathcal{A}^0_\mathcal{C}(E_j)(X)=L_j(L_j^\star G_j(\mathcal{A}^0_\mathcal{C}(E_j)(X)))\oplus \mathcal{H}_{\Delta_j}(E_j)\oplus L_j^\star (L_jG_j(\mathcal{A}^0_\mathcal{C}(E_j)(X)))
$$
\item the following relations hold:
\begin{itemize}
\item $G_j\Delta_j+H_j=\Delta_jG_j+H_j=id|_{\Gamma(X,E_j)}$
\item $H_jG_j=G_jH_j=H_j\Delta_j=\Delta_jH_j=0$
\item $L_j\Delta_j=\Delta_jL_j; L^\star_j\Delta_j=\Delta_jL^\star_j$
\end{itemize}
\item $dim_\mathbb{R}\mathcal{H}_{\Delta_j}(E_j)<\infty$ and there is a canonical isomorphism $\mathcal{H}_{\Delta_j}(E_j)\simeq H^j(E)$.
\end{enumerate}
\end{theorem}

\section{Hodge decomposition theorem}\label{section4}
The principle underlying the proof of the Hodge decomposition theorem is that every de Rham cohomology 
class uniquely represents a harmonic form for an elliptic operator, namely the Laplacian. 
The K\"ahler condition plays a pivotal role in the proof and leads to a decomposition of harmonic forms 
into components of type $(p, q)$.

For a given smooth $d$-Hermitian metric on a $d$-complex manifold $(X,\dholos{X})$, we have a smooth $d$-complex $(1,1)$ form $\omega$ taking values in the line bundle $\mathcal{L}$ (the line bundle as described in Subsection \ref{subsec2.1}), which can be expressed as a family 
$$
\left\lbrace\omega_{U_l}=\frac{i}{2}\displaystyle\sum_{i_1,j_1=1}^n  h^{U_l}_{i_1j_1}dz^{i_1}_l\wedge d\overline{z}^{j_1}_l\right\rbrace_{l\in I}
$$ 
with respect to the atlas $\mathcal{U}=\{(U_i,\phi_i)\}_{i\in I}$ on $d$-complex manifold $(X,\dholos{X})$, satisfying the compatibility condition
\[
\omega_{U_{l'}}=\left\{\begin{array}{lr}
        \omega_{U_l}, & \text{ if }\trans{l}{l'}\text{ is holomorphic};\vspace{7pt}\\
        -\overline{\omega}_{U_l}, & \text{ if }\trans{l}{l'}\text{ is anti-holomorphic}.
        \end{array}\right.
\]
The smooth $d$-complex $(1,1)$ form $\omega$ is the fundamental form associated with the $d$-Hermitian metric. 
\begin{definition}\rm{
The $d$-Hermitian manifold $(X,\dholos{X})$ is called $d$-K\"{a}hler, denoted by $(X,\omega)$ iff the smooth $d$-complex $(1,1)$ form $\omega$ is $d_\mathcal{C}$-closed i.e. $d(\omega_{U_i})=0$ for each $i\in I$.
}
\end{definition}
\begin{theorem}\label{fundaform}
Let $(X,\omega)$ be a $d$-K\"ahler manifold of complex dimension $n$ with $d$-K\"ahler form $\omega=\left\{\omega_{U_l}=\frac{i}{2}\displaystyle\sum_{\alpha,\beta}h^{U_l}_{\alpha\beta}dz^\alpha_l\wedge d\overline{z}^\beta_l\right\}_{l\in I}$. Then around each point $x\in X$, we can choose co-ordinates $\phi_{l_x}=(z_{l_x}^1,z_{l_x}^2,\dots,z_{l_x}^n)$ such that $h^{U_l}_{jk}=\delta^j_k+O(|z|^2)$.
\end{theorem}
\begin{proof}
On a chart $\big(U_{l_x},\phi_{l_x}=(z_{l_x}^1,z_{l_x}^2,\dots,z_{l_x}^n)\big)$ around a point $x\in X$, we have given $\omega_{U_{l_x}}=\frac{i}{2}\displaystyle\sum_{\alpha,\beta}h^{U_{l_x}}_{\alpha\beta}dz^\alpha_{l_x}\wedge dz^\beta_{l_x}$, where $[h^{U_{l_x}}_{\alpha\beta}]$ is symmetric, positive definite matrix. Using taylor series expansion of $h^{U_{l_x}}_{\alpha\beta}$ around the point $\phi_{l_x}(x)=0$, we have
$$
h^{U_{l_x}}_{\alpha\beta}(z)=h^{U_{l_x}}_{\alpha\beta}(0)+\displaystyle\sum_i\frac{\partial h^{U_{l_x}}_{\alpha\beta}}{\partial z_{l_x}^i}(0)z_{l_x}^i+\displaystyle\sum_i\frac{\partial h^{U_{l_x}}_{\alpha\beta}}{\partial \overline{z}_{l_x}^i}(0)\overline{z}_{l_x}^i+ O(|z|^2)
$$
Using Gram-Schmidt orthonormalization, we can choose $(z_{l_x}^1,z_{l_x}^2,\dots,z_{l_x}^n)$ such that $h^{U_{l_x}}_{\alpha\beta}(0)\equiv \delta^\alpha_\beta$; hence we have
\begin{align}
h^{U_{l_x}}_{\alpha\beta}(z)&=\delta^\alpha_\beta+\displaystyle\sum_i\frac{\partial h^{U_{l_x}}_{\alpha\beta}}{\partial z_{l_x}^i}(0)z_{l_x}^i+\displaystyle\sum_i\frac{\partial h^{U_{l_x}}_{\alpha\beta}}{\partial \overline{z}_{l_x}^i}(0)\overline{z}_{l_x}^i+ O(|z|^2)\quad\text{and}\nonumber\\ 
\omega_{U_{l_x}}&=\frac{i}{2}\displaystyle\sum_{\alpha,\beta}h_{\alpha\beta}^{U_{l_x}}(z) dz^\alpha_{l_x}\wedge d\overline{z}^\beta_{l_x}\label{formomeg}\nonumber\\
&=\frac{i}{2}\displaystyle\sum_{\alpha,\beta}\left(\delta^\alpha_\beta+\displaystyle\sum_i\frac{\partial h^{U_{l_x}}_{\alpha\beta}}{\partial z_{l_x}^i}(0)z_{l_x}^i+\displaystyle\sum_i\frac{\partial h^{U_{l_x}}_{\alpha\beta}}{\partial \overline{z}_{l_x}^i}(0)\overline{z}_{l_x}^i\right) dz^\alpha_{l_x}\wedge d\overline{z}^\beta_{l_x}+O(|z|^2)
\end{align}
Since $d\omega_{U_{l_x}}=0$, we have
\begin{align}
\frac{\partial h_{\alpha i}^{U_{l_x}}}{\partial z_{l_x}^\beta}(z)&=\frac{\partial h_{\beta i}^{U_{l_x}}}{\partial z_{l_x}^\alpha}(z)\quad\text{and}\label{eq1}\\
\frac{\partial h_{\alpha i}^{U_{l_x}}}{\partial \overline{z}_{l_x}^\beta}(z)&=\frac{\partial h_{\alpha \beta}^{U_{l_x}}}{\partial \overline{z}_{l_x}^i}(z)\quad(\text{for }\alpha,\beta,i\in \{1,2,\dots,n\})\label{eq2}
\end{align}
Change the co-ordinates as
\begin{align*}
\tilde{z}^i_{l_x}=z^i_{l_x}+\frac{1}{2}\displaystyle\sum_{\alpha,\beta} \frac{\partial h_{\beta i}^{U_{l_x}}}{\partial z^\alpha_{l_x}}(0)z^\alpha_{l_x} z^\beta_{l_x}.
\end{align*}
Which gives
\begin{align*}
d\tilde{z}^i_{l_x}&=dz^i_{l_x}+\frac{1}{2}\left(\displaystyle\sum_{\alpha,\beta}\frac{\partial h^{U_{l_x}}_{\beta i}}{\partial z_{l_x}^\alpha}(0)z_{l_x}^\alpha dz_{l_x}^\beta+\frac{\partial h^{U_{l_x}}_{\beta i}}{\partial z_{l_x}^\alpha}(0)dz_{l_x}^\alpha z_{l_x}^\beta\right)\\
&=dz^i_{l_x}+\frac{1}{2}\displaystyle\sum_{\alpha,\beta}\left(\frac{\partial h_{\beta i}^{U_{l_x}}}{\partial z_{l_x}^\alpha}(0)+\frac{\partial h^{U_{l_x}}_{\alpha i}}{\partial z_{l_x}^\beta}(0)\right)z_{l_x}^\alpha dz_{l_x}^\beta\\
&=dz^i_{l_x}+\displaystyle\sum_{\alpha,\beta}\frac{\partial h^{U_{l_x}}_{\beta i}}{\partial z_{l_x}^\alpha}(0)z_{l_x}^\alpha dz_{l_x}^\beta\quad(\text{using Equation }\eqref{eq1}).
\end{align*}
Similarly, using Equation \eqref{eq2}, we have
$$
d\overline{\tilde{z}}^i_{l_x}=d\overline{z}^i_{l_x}+\displaystyle\sum_{\alpha,\beta}\frac{\partial h^{U_{l_x}}_{\beta i}}{\partial \overline{z}_{l_x}^\alpha}(0)\overline{z}_{l_x}^\alpha d\overline{z}_{l_x}^\beta
$$
Furthermore, we have
\begin{align*}
d\tilde{z}^i_{l_x}\wedge d\overline{\tilde{z}}^i_{l_x}=dz^i_{l_x}\wedge d\overline{z}^i_{l_x}+\displaystyle\sum_{\alpha,\beta}\frac{\partial h^{U_{l_x}}_{\beta i}}{\partial \overline{z}^\alpha_{l_x}}(0)\overline{z}^\alpha_{l_x} dz^i_{l_x}\wedge d\overline{z}^\beta_{l_x}+\displaystyle\sum_{\alpha,\beta}\frac{\partial h^{U_{l_x}}_{\beta i}}{\partial z^\alpha_{l_x}}(0)z^\alpha_{l_x} dz^\beta_{l_x}\wedge d\overline{z}^i_{l_x}+O(|z|^2)
\end{align*}
which further gives
\begin{align}
\frac{i}{2}\displaystyle\sum_i d\tilde{z}^i_{l_x}\wedge d\overline{\tilde{z}}^i_{l_x}&=\frac{i}{2}\left(\displaystyle\sum_idz^i_{l_x}\wedge d\overline{z}^i_{l_x}+\displaystyle\sum_{i,\alpha,\beta}\frac{\partial h^{U_{l_x}}_{\beta i}}{\partial \overline{z}^\alpha_{l_x}}(0)\overline{z}^\alpha_{l_x} dz^i_{l_x}\wedge d\overline{z}^\beta_{l_x}+\displaystyle\sum_{i,\alpha,\beta}\frac{\partial h^{U_{l_x}}_{\beta i}}{\partial z^\alpha_{l_x}}(0)z^\alpha_{l_x} dz^\beta_{l_x}\wedge d\overline{z}^i_{l_x}\right)\nonumber\\
&\qquad\quad+O(|z|^2)\nonumber\\
&=\frac{i}{2}\displaystyle\sum_{\beta,i}\left(\delta^\beta_i+\displaystyle\sum_{\alpha}\frac{\partial h^{U_{l_x}}_{i \beta}}{\partial \overline{z}^\alpha_{l_x}}(0)\overline{z}^\alpha_{l_x} +\displaystyle\sum_{\alpha}\frac{\partial h^{U_{l_x}}_{\beta i}}{\partial z^\alpha_{l_x}}(0)z^\alpha_{l_x} \right)dz^\beta_{l_x}\wedge d\overline{z}^i_{l_x}+O(|z|^2)\nonumber\\
&=\frac{i}{2}\displaystyle\sum_{\alpha,\beta}\left(\delta^\alpha_\beta+\displaystyle\sum_{i}\frac{\partial h^{U_{l_x}}_{\beta \alpha}}{\partial \overline{z}^i_{l_x}}(0)\overline{z}^i_{l_x} +\displaystyle\sum_{i}\frac{\partial h^{U_{l_x}}_{\alpha \beta}}{\partial z^i_{l_x}}(0)z^i_{l_x} \right)dz^\alpha_{l_x}\wedge d\overline{z}^\beta_{l_x}+O(|z|^2)\nonumber\\
&=\frac{i}{2}\displaystyle\sum_{\alpha,\beta}\left(\delta^\alpha_\beta+\displaystyle\sum_{i}\frac{\partial h^{U_{l_x}}_{\alpha \beta}}{\partial z^i_{l_x}}(0)z^i_{l_x} + \displaystyle\sum_{i}\frac{\partial h^{U_{l_x}}_{\alpha \beta}}{\partial \overline{z}^i_{l_x}}(0)\overline{z}^i_{l_x} \right)dz^\alpha_{l_x}\wedge d\overline{z}^\beta_{l_x}+O(|z|^2)\label{formomeg1}
\end{align}
Using equations \eqref{formomeg} and \eqref{formomeg1}, we have
\begin{align*}
\frac{i}{2}\displaystyle\sum_{\alpha,\beta}\delta^\alpha_\beta d\tilde{z}^\alpha_{l_x}\wedge d\overline{\tilde{z}}^\beta_{l_x}=\omega_{U_{l_x}}+O(|z|^2)
\end{align*}
Hence, we have the required proof of the theorem.
\end{proof}
Let $(X,\dholos{X})$ be a compact $d$-Hermitian manifold of real dimension $2n$. The Riemannian metric on the associated smooth manifold $X$ induces a Riemannian inner product on the bundle $\wedge^pT_X$ (respectively, $\wedge^pT_X\otimes \mathcal{L}$) as well as on its dual $\wedge^pT^\star_X $ (respectively, $\wedge^pT^\star_X \otimes \mathcal{L}$) and using the same line of arguments as in Section \ref{section2}, we have an inner product on the space $\mathcal{A}^k(X)$ (respectively, $\mathcal{A}^k(\mathcal{L})(X)$) $(1\leq k\leq 2n)$. There is a dual operator
$$
\star:\mathcal{A}^k(X)\rightarrow \mathcal{A}^{2n-k}(\mathcal{L})(X)
$$
such that for $\alpha,\beta\in \mathcal{A}^p(X)$ $(1\leq p\leq 2n)$, we have
$$
\alpha\wedge \star \beta=\langle \alpha,\beta\rangle (\star 1).
$$
\begin{definition}\rm{
As described above, the dual operator $\star$ is called the Hodge star operator.
}
\end{definition}
Using the Hodge star operator, we describe the formal adjoint $d^\star$\big(as well as $d^\star_\mathcal{L}$\big). Considering the manifold $X$ without boundary and using Stokes' theorem, the formal adjoint can be expressed as $d^\star=-\star d_\mathcal{L} \star$, as described below.

Let $\alpha\in \mathcal{A}^k(X)$, $\beta\in \mathcal{A}^{k+1}(X)$. By definition, we have
$$
(d\alpha,\beta)=\displaystyle\int_X d\alpha\wedge \star \beta.
$$
Since, $\alpha\wedge \star \beta\in \mathcal{A}^{2n-1}(\mathcal{L})(X)$, we have
$$
d_\mathcal{L}(\alpha\wedge \star \beta)=d\alpha\wedge \star \beta+(-1)^k\alpha\wedge d_\mathcal{L}(\star\beta).
$$
If $X$ is a manifold without boundary, by Stokes' theorem
$$
\displaystyle\int_X d_\mathcal{L}(\alpha\wedge\star\beta)=0,
$$
which implies
$$
\displaystyle\int_X d\alpha\wedge \star \beta=(-1)^{k+1}\displaystyle\int_X\alpha\wedge d_\mathcal{L}(\star \beta).
$$
Note that 
\begin{align*}
d_\mathcal{L}(\star \beta)&=(-1)^{k(2n-k)}\star(\star d_\mathcal{L}(\star \beta))\quad(\text{Since, }\star\star(\alpha)=(-1)^{k(2n-k)}\alpha\text{ for each k-form }\alpha)\\
&=(-1)^k\star(\star d_\mathcal{L}(\star \beta))\\
\text{or, }(-1)^k d_\mathcal{L}(\star \beta)&=\star (\star d_\mathcal{L}(\star \beta)).
\end{align*}
Which further implies
\begin{align*}
(d(\alpha),\beta)&=\displaystyle\int_X d(\alpha)\wedge \star \beta\\
&=-\displaystyle\int_X \alpha\wedge(-1)^kd_\mathcal{L}(\star \beta)\\
&=-\displaystyle\int_X \alpha \wedge \star(\star d_\mathcal{L}(\star \beta))\\
&=(\alpha,-\star d_\mathcal{L}(\star \beta)).
\end{align*}
Hence, we have $(\alpha,d^\star \beta)=(\alpha,-\star(d_\mathcal{L}(\star \beta)))$.

 There is also an induced inner product on the space of $d$-complex smooth $k$-forms $\mathcal{A}^k_\mathcal{C}(X)$ as described in Section \ref{section2}. We can extend the definition of dual operator $\star$ on $\mathcal{A}_\mathcal{C}^k(X)$ by defining
$$
\star:\mathcal{A}_\mathcal{C}^k(X)\rightarrow \mathcal{A}^{2n-k}_\mathcal{C}(X)\quad (\alpha_1+i\alpha_2\mapsto \star(\alpha_2)+i\star(\alpha_1)).
$$
Using the decomposition $\mathcal{A}_\mathcal{C}^k(X)=\displaystyle\sum_{\substack{p+q=k \\ 1\leq p,q\leq n}} \mathcal{A}^{(p,q)}(X)$, the Hodge star operator sends a $(p,q)$ form to $(n-p,n-q)$ form for $1\leq p,q\leq n$. For two elements $\alpha,\beta\in \mathcal{A}^{(p,q)}(X)$ the Hodge star operator satisfies
$$
\alpha\wedge \star \beta=\langle \alpha,\beta \rangle (\star 1')\quad(\text{where, }\star 1'=\text{either }\star 1\text{ or, }\star 1_\mathcal{L})
$$

As in the case of usual de-Rham operator, we also have adjoints for the operators $d_\mathcal{C},\partial$ and $\overline{\partial}$, which can be expressed as $d_\mathcal{C}^\star=-\star d_\mathcal{C}\star,\partial^\star=-\star \overline{\partial}\star$ and $\overline{\partial}^\star=-\star \partial\star$. 

Define the Laplacians 
\begin{align*}
\Delta_{d_\mathcal{C}}&=d_\mathcal{C}\circ d_\mathcal{C}^\star+d_\mathcal{C}^\star\circ d_\mathcal{C},\\ 
\Delta_\partial&=\partial \circ \partial^\star+\partial^\star\circ \partial\text{ and }\\
\Delta_{\overline{\partial}}&=\overline{\partial}\circ \overline{\partial}^\star+\overline{\partial}^\star\circ \overline{\partial}.
\end{align*} 
The above-described Laplacians are self-adjoint elliptic operators of order 2, degree 0 (for details, see \cite[Chapter 4, Section 5, Page no. 145]{ROW}). The space of Harmonic forms, that is, the  kernel of Laplacians, will be written as
\begin{align*}
\mathcal{H}_{\Delta_{d_\mathcal{C}}}(\wedge^k{T^\star}^\mathcal{C}_X X)&=\mathcal{H}^k(X)_\mathcal{C},\\
\mathcal{H}_{\Delta_\partial}(\mathcal{A}^{(p,q)}(X))&=\mathcal{H}^{(p,q)}_\partial(X)\text{ and }\\
\mathcal{H}_{\Delta_{\overline{\partial}}}(\mathcal{A}^{(p,q)}(X))&=\mathcal{H}^{(p,q)}_{\overline{\partial}}(X).
\end{align*}

Let $(X,\omega)$ be a $d$-K\"{a}hler manifold, where $\omega\in \mathcal{A}^{1,1}(\mathcal{L})(X)$ is $d_\mathcal{C}$-closed, we have an operator
$$
\mathfrak{L}:\mathcal{A}^k_\mathcal{C}(X)\rightarrow \mathcal{A}^{k+2}_\mathcal{C}(\mathcal{L})(X)\quad (\alpha\mapsto \omega \wedge \alpha)
$$
of degree 2, called the Lefschetz operator along with its adjoint $\Lambda:\mathcal{A}^{k+2}_\mathcal{C}(X)\rightarrow \mathcal{A}^k_\mathcal{C}(\mathcal{L})(X)$.
\begin{theorem}
The adjoint operator $\Lambda:\mathcal{A}^{k+2}_\mathcal{C}(X)\rightarrow \mathcal{A}^k_\mathcal{C}(\mathcal{L})(X)$ is given by
$$
\Lambda(\beta)=(-1)^k\star \mathfrak{L} \star (\beta)
$$
for every $\beta\in \mathcal{A}^{k+2}_\mathcal{C}(X)$.
\end{theorem}
\begin{proof}
Let $\alpha\in \mathcal{A}^k_\mathcal{C}(X), \beta\in \mathcal{A}^{k+2}_\mathcal{C}(X)$, then
\begin{align*}
\langle \mathfrak{L}(\alpha),\beta \rangle (\star 1')&= \mathfrak{L}(\alpha)\wedge \star \beta\\
&=(\omega \wedge \alpha)\wedge \star \overline{\beta}\\
&=(-1)^{2k}\alpha\wedge(\omega\wedge \star\overline{\beta})\\
&=\alpha\wedge(\omega\wedge \star\overline{\beta}).
\end{align*}
Note that
\begin{align*}
(\omega\wedge \star\overline{\beta})&=(-1)^{k(2n-k)}\star(\star(\omega\wedge \star(\overline{\beta})))\\
&=(-1)^k\star(\overline{\star(\omega\wedge \star(\beta))})\quad (\text{because }\omega\text{ is real form and }\star\text{ is $\mathbb{R}$-linear})\\
&=(-1)^k\star(\overline{\star(\mathfrak{L}(\star \beta))}).
\end{align*}
Further, we have
\begin{align*}
\langle \mathfrak{L}(\alpha),\beta \rangle (\star 1')&=\alpha\wedge (-1)^k \star(\overline{\star(\mathfrak{L}(\star \beta))})\\
&=\langle \alpha, (-1)^k\star(\mathfrak{L}(\star \beta))\rangle (\star 1').
\end{align*}
This proves the proposition.
\end{proof}
\begin{theorem}\label{Kahleriden}
If $(X,\omega)$ is $d$-K\"{a}hler manifold, then we have 
\begin{align}
[\overline{\partial}^\star, \mathfrak{L}]&=i\partial,\label{4.1}\\
[\partial^\star, \mathfrak{L}]&=-i\overline{\partial},\label{4.2}\\
[\Lambda, \overline{\partial}]&=-i\partial^\star,\label{4.3}\\
[\Lambda,\partial]&=i\overline{\partial}^\star.\label{4.4}
\end{align}
\end{theorem}
\begin{proof}
Let $\alpha\in \mathcal{A}^q(\mathbb{C}^n)$ and $\beta=\mathcal{A}^{q-1}(\mathbb{C}^n)$, define $dz^j\vee \alpha\in \mathcal{A}^{q-1}(\mathbb{C}^n)$, keep ensuring that
$$
\langle dz^j\vee \alpha,\beta\rangle=\langle \alpha,d\overline{z}^j\wedge \beta\rangle.
$$
Let $\alpha=dz^j,\beta=1$ (for $q=1$) then
\begin{align*}
\langle dz^j\vee dz^j,1\rangle=\langle dz^j,d\overline{z}^j\rangle.
\end{align*}
We have $dz^j\vee dz^j=\langle dz^j,d\overline{z}^j\rangle$.

Similarly for $k\in \{1,2,\dots,n\}$, we have
\begin{align*}
dz^j\vee dz^k&=\langle dz^k,d\overline{z}^j\rangle\quad\text{and}\\
    dz^j\vee d\overline{z}^k&=\langle d\overline{z}^k,d\overline{z}^j\rangle .
\end{align*}

Let $\alpha\in \mathcal{A}^q(\mathbb{C}^n)$ and $\beta\in \mathcal{A}^{q-1}(\mathbb{C}^n)$ such that $\beta$ has compact support, which implies the smooth function $\langle dz^j\vee \alpha,\beta\rangle$ has compact support. Note that
\begin{align*}
\displaystyle\int_{\mathbb{C}^n}\frac{\partial \left\langle dz^j\vee \alpha,\beta\right\rangle}{\partial z^j} (\star 1)&=\displaystyle\int_{\mathbb{C}^n}\left\langle dz^j\vee \frac{\partial \alpha}{\partial z^j},\beta\right\rangle(\star 1)\\
&\qquad+\displaystyle\int_{\mathbb{C}^n}\left\langle dz^j\vee \alpha,\frac{\partial \beta}{\partial \overline{z}^j}\right\rangle(\star 1)\quad(\star 1\text{ is standard volume form on }\mathbb{C}^n)\\
&=\left(dz^j\vee \frac{\partial \alpha}{\partial z^j},\beta\right)+\left(\alpha,d\overline{z}^j\wedge \frac{\partial \beta}{\partial \overline{z}^j}\right)\quad(\text{for each }j\in \{1,2,\dots,n\}).
\end{align*}
Furthermore, we have
\begin{align*}
\left(\displaystyle\sum_j dz^j\vee \frac{\partial \alpha}{\partial z^j},\beta\right)+\left(\alpha,d\overline{z}^j\wedge\displaystyle\sum \frac{\partial \beta}{\partial \overline{z}^j}\right)&=\left(\displaystyle\sum_j dz^j\vee \frac{\partial \alpha}{\partial z^j},\beta\right)+\left(\alpha,\overline{\partial}\beta\right)\\
&=\left(\displaystyle\sum_j dz^j\vee \frac{\partial \alpha}{\partial z^j},\beta\right)+\left(\overline{\partial}^\star\alpha,\beta\right).
\end{align*}
Using Stokes' theorem and the above calculation, we have
\begin{equation}\label{state1}
\overline{\partial}^\star\equiv -\displaystyle\sum_j dz^j\vee \frac{\partial}{\partial z^j}.
\end{equation}
Let $u\in \mathcal{A}^{(p,q)}(X)$, then
\begin{align*}
[\overline{\partial}^\star,\mathfrak{L}](u)&=\overline{\partial}^\star (\mathfrak{L}(u))-\mathfrak{L}(\overline{\partial}^\star(u))\\
&=\overline{\partial}^\star (\omega\wedge u)-\omega\wedge (\overline{\partial}^\star(u)).
\end{align*}
Let $x_0\in X$ be an arbitrary point and by Theorem \ref{fundaform}, we can have a local co-ordinate chart around $x_0$ say, $(U_{x_0},\tilde{\phi}_{x_0}=(\tilde{z}_{x_0}^1,\tilde{z}_{x_0}^2,\dots,\tilde{z}_{x_0}^n))$ such that 
$$
\omega_{U_{x_0}}(\phi_{x_0}(x))=\displaystyle\sum_{\alpha\beta}\delta^\alpha_\beta d\tilde{z}^\alpha_{x_0}\wedge d\overline{\tilde{z}}^\beta_{x_0}+O(|\phi_{x_0}(x)|^2)\quad(\text{for }x\in U_{x_0})
$$
Since in the commutator $[\overline{\partial}^\star,\mathfrak{L}]$ only first order partial differentials appear, while the first order partial derivatives of entries of Hermitian matrix associated to $d$-K\"ahler form vanishes at the point in transformed co-ordinates as described above, so it is sufficient to give the required proof of the theorem, for the standard $d$-K\"ahler form
$$
\omega=\frac{i}{2}\displaystyle\sum_{\alpha\beta}\delta^\alpha_\beta dz^\alpha\wedge d\overline{z}^\beta
$$

Using Equation \eqref{state1},
\begin{align*}
\overline{\partial}^\star(\omega\wedge\alpha)&=-\displaystyle\sum_j dz^j\vee \frac{\partial (\omega\wedge \alpha)}{\partial z^j}\\
&=-\displaystyle\sum dz^j\vee \left(\omega\wedge \frac{\partial \alpha}{\partial z^j}\right)\quad(\text{as }\frac{\partial \omega}{\partial z^j}=0)\\
&=\frac{-i}{2}\displaystyle\sum_j dz^j\vee \left(\displaystyle\sum_k dz^k\wedge d\overline{z}_i\wedge \frac{\partial \alpha}{\partial z^j}\right)
\end{align*}
\begin{align*}
&=\frac{-i}{2}\displaystyle\sum_j\left(\displaystyle\sum_k(dz^j\vee dz^k)\wedge d\overline{z}^k\wedge \frac{\partial \alpha}{\partial z^j}\right)\\
&\qquad+\frac{i}{2}\displaystyle\sum_j\left(\displaystyle\sum_k dz^k\wedge (dz^j\vee d\overline{z}^k)\wedge \frac{\partial \alpha}{\partial z^j}\right)\\
&\qquad\qquad\frac{-i}{2}\displaystyle\sum_j\left(\displaystyle\sum_k dz^k\wedge d\overline{z}^k\wedge \left(dz^j\vee \frac{\partial \alpha}{\partial z^j}\right)\right)\\
&=i\displaystyle\sum_j\left(dz^j\wedge \frac{\partial \alpha}{\partial z^j}\right)+\omega \wedge \left(\displaystyle\sum_j dz^j\vee \frac{\partial \alpha}{\partial z^j}\right)\\
&=i\displaystyle\sum_j\left(dz^j\wedge \frac{\partial \alpha}{\partial z^j}\right)+\omega \wedge (\overline{\partial}^\star \alpha)
\end{align*}
Which further gives,
\begin{align*}
[\overline{\partial}^\star,\mathfrak{L}]\alpha&=i\partial(\alpha)\\
\text{that means, }[\overline{\partial}^\star,\mathfrak{L}]&= i\partial\quad(\text{as required})
\end{align*}
Hence, we get \eqref{4.1} and by taking conjugate of \eqref{4.1}, we have \eqref{4.2}. Similarly, we get \eqref{4.3} from \eqref{4.4}. The equation \eqref{4.4} follows from \eqref{4.1} because of formal adjoint property.
\end{proof}
\begin{theorem}\label{lapequiv}
Let $(X,\omega)$ be a $d$-K\"{a}hler manifold and $\Delta_{d_\mathcal{C}},\Delta_\partial,\Delta_{\overline{\partial}}$ be the Laplacians of the operators $d_\mathcal{C},\partial,\overline{\partial}$, respectively. Then we have
$$
\Delta_{d_\mathcal{C}}=2\Delta_\partial=2\Delta_{\overline{\partial}}\hspace{2pt}.
$$
In particular, an element $\alpha\in \mathcal{A}^k_\mathcal{C}(X)$ is $\Delta_{d_\mathcal{C}}$-harmonic if and only if it is $\Delta_\partial$-harmonic if and only if it is $\Delta_{\overline{\partial}}$-harmonic.
\end{theorem}
\begin{proof}
It follows from theorem \ref{Kahleriden} and the fact that $d=\partial+\overline{\partial}$.
\end{proof}
\begin{cor}\label{lapequivcor}
Let $(X,\omega)$ be a $d$-K\"{a}hler manifold then for $\alpha\in \mathcal{A}^{(p,q)}(X)$, $\Delta_{d_\mathcal{C}}(\alpha)\in \mathcal{A}^{(p,q)}(X)$.
\end{cor}
\begin{proof}
This is obvious using the Theorem \ref{lapequiv}.
\end{proof}
Since, $\Delta_\partial=\Delta_{\overline{\partial}}$, we have $\mathcal{H}_\partial^{(p,q)}(X)=\mathcal{H}_{\overline{\partial}}^{(p,q)}(X)$, which will be denoted by $\mathcal{H}^{(p,q)}(X)$ and using Theorem \ref{lapequiv} and Corollary \ref{lapequivcor}, we have the following:

\begin{theorem}\label{hodgedecomp}
Let $(X,\omega)$ be a $d$-K\"{a}hler manifold and $\alpha=\displaystyle\sum_{p+q=k}\alpha^{(p,q)}$ be the decomposition of a $d$-smooth $k$-form on $X$. Then $\alpha$ is $\Delta_{d_\mathcal{C}}$-harmonic iff $\alpha^{(p,q)}$ is $\Delta_\partial$-harmonic iff $\alpha^{(p,q)}$ is $\Delta_{\overline{\partial}}$-harmonic. In particular,
$$
\mathcal{H}^k(X)_\mathcal{C}=\displaystyle\bigoplus_{p+q=k}\mathcal{H}^{(p,q)}(X).
$$
\end{theorem}

Note that the de-Rham chain complex, as well as the Dolbeault chain complex, are elliptic complexes (see Section \ref{section2}) and their Laplacians are differential operators of order $2k$, which are self-adjoint, and hence
both are elliptic differential operators \cite[Chapter 4, Section 5, Page no. 145]{ROW}. Therefore, by using Theorem \ref{ellipticdec} for both  elliptic complexes \eqref{derhamcomp} and \eqref{dolbeaultcomp}, we have

\begin{theorem}\label{formdecomp}
Let $X$ be a compact $d$-Hermitian manifold. Then, there exist an orthogonal decompositions,
$$
\mathcal{A}^k_\mathcal{C}(X)=d_\mathcal{C}\big(\mathcal{A}^{k-1}_\mathcal{C}(X)\big)\oplus \mathcal{H}^k(X)_\mathcal{C}\oplus d_\mathcal{C}^\star\big(\mathcal{A}^{k+1}_\mathcal{C}(X)\big).
$$
Furthermore, $\mathcal{H}^k(X)_\mathcal{C}$ is finite dimensional along with 
$$
H^k_{dR}(X,\mathcal{C})\simeq \mathcal{H}^k(X)_\mathcal{C}.
$$ 
\end{theorem}

\begin{theorem}\label{formdecomp1}
Let $X$ be a compact $d$-Hermitian manifold of real dimension $2n$. Then, there exist orthogonal decompositions,
\begin{align*}
\mathcal{A}^{(p,q)}(X)&=\partial\big(\mathcal{A}^{p-1,q}_\mathcal{C}(X)\big)\oplus \mathcal{H}^{(p,q)}_\partial
(X)\oplus \partial^\star\big(\mathcal{A}^{p+1,q}_\mathcal{C}(X)\big)\quad\text{ and}\\
\mathcal{A}^{(p,q)}(X)&=\overline{\partial}\big(\mathcal{A}^{p,q-1}_\mathcal{C}(X)\big)\oplus \mathcal{H}^{(p,q)}_{\overline{\partial}}
(X)\oplus \overline{\partial}^\star\big(\mathcal{A}^{p,q+1}_\mathcal{C}(X)\big).
\end{align*}
Furthermore, $\mathcal{H}^k_d,\mathcal{H}_\partial^{(p,q)}$ and $H^{(p,q)}_{\overline{\partial}}$ are finite dimensional along with 
\begin{align*}
H^{(p,q)}_{dB}(X)&\simeq \mathcal{H}^{(p,q)}_\partial\quad\text{ and}\\
H^{(p,q)}_{dB}(X)&\simeq\mathcal{H}^{(p,q)}_{\overline{\partial}}.
\end{align*} 
for all $1\leq p,q\leq n\in \mathbb{Z}$.
\end{theorem}
\begin{cor}\label{vanish}
Let $X$ be a compact $d$-complex manifold with $d$-Hermitian metric $h$, then $Im(\partial^\star)\cap Ker(\partial)=\{0\}$
\end{cor}
\begin{proof}
Let $\alpha\in Im(\partial^\star)\cap Ker(\partial)$, we have $\partial(\alpha)=0=\partial^\star(\alpha)$, furthermore $\alpha\in \mathrm{Ker}(\Delta_\partial)$ i.e. $\alpha$ is a $\Delta_\partial$-harmonic form but using the orthogonal decomposition in the theorem \ref{formdecomp1}, we have $\alpha=0$.
\end{proof}
Using Theorems \ref{hodgedecomp},\ref{formdecomp} and \ref{formdecomp1} and the fact that $\mathcal{H}_{\overline{\partial}}^{(p,q)}=\mathcal{H}_\partial^{(p,q)}=\mathcal{H}^{(p,q)}$ for a compact $d$-K\"{a}hler manifold, we have the following: 

\begin{theorem}[Hodge decomposition]\label{mainthm-hodge}
Let $(X,\dholos{X})$ be a compact $d$-K\"{a}hler manifold, then we have the direct sum decomposition,
$$
H^k_{dR}(X,\mathcal{C})=\displaystyle\bigoplus_{p+q=k}H^{(p,q)}_{dB}(X).
$$
\end{theorem}

\subsection*{Example}
Note that the Klein surfaces are $d$-K\"{a}hler manifolds. Let us now discuss examples of non $d$-K\"{a}hler non-orientable $d$-complex manifolds.

\begin{definition}\rm{
A compact complex manifold of complex dimension 2, which has universal covering biholomorphic to $W=\mathbb{C}^2-\{(0,0)\}$ and fundamental group isomorphic to $\mathbb{Z}$, is called a primary Hopf surface.

A complex analytic automorphism $f$ on $\mathbb{C}^2$ is called a contraction if $\displaystyle{\lim_{n\to \infty}f^n(B)}=0$, where $B=\{(z_1,z_2)\in \mathbb{C}^2:|z_1|^2+|z_2|^2\leq 1\}\subset \mathbb{C}^2$.
}
\end{definition}

Any primary Hopf surface can be expressed as $\bigslant{\big(\mathbb{C}^2- \{(0,0)\}\big)}{\langle f\rangle}$, where $f:\mathbb{C}^2\rightarrow \mathbb{C}^2$ is a contraction such that $f(0,0)=(0,0)$. 

We can choose an appropriate co-ordinates on $\mathbb{C}^2$, the contraction $f$ has the normal form:
\begin{align*}
f:\mathbb{C}^2- \{(0,0)\}&\longrightarrow \mathbb{C}^2-\{(0,0)\}\\
(z_1,z_2)&\mapsto \left(\frac{\alpha_1z_1+\lambda z_2^m}{\alpha_2z_2}\right)
\end{align*} 
where $m\in \mathbb{Z}^+$ and $\alpha_1,\alpha_2,\lambda$ are constants with the following conditions:
$$
(\alpha_1-\alpha_2^m)\lambda=0,\quad 0<|\alpha_1|\leq |\alpha_2|<1.
$$
For more details on Hopf surfaces, see \cite{KK}.

A primary Hopf surface $\big(\bigslant{\mathbb{C}^2- \{(0,0)\}\big)}{\langle f\rangle}$ can be written as 
$$
H_f=\bigslant{W}{\langle f\rangle}=\bigslant{\big(\mathbb{C}^2- \{(0,0)\}\big)}{\langle f\rangle}
$$ 
with canonical projection map $\pi:W\rightarrow H_f$ for some contraction $f$ on $\mathbb{C}^2- \{(0,0)\}$. In \cite{KK1}, Kunhiko Kodaira has proved that a Hopf surface is topologically homeomorphic to $S^1\times S^3$.

Using the theory of covering space, any holomorphic (respectively, anti-holomorphic) topological morphism $\sigma:H_{f_1}=\bigslant{W}{\langle f_1 \rangle}\rightarrow H_{f_2}=\bigslant{W}{\langle f_2 \rangle}$ between two Hopf surfaces is induced by a holomorphic (respectively, anti-holomorphic) topological automorphism $\widehat{\sigma}$ on $W=\mathbb{C}^2- \{(0,0)\}$ such that $\pi_{f_2}\circ \hat{\sigma}\equiv \sigma\circ \pi_{f_1}$ 

\begin{definition}\rm{
An anti-holomorphic involution on a Hopf surface $H_f$ is an anti-holomorphic topological automorphism $\sigma: H_f\rightarrow H_f$ such that $\sigma^2\equiv id_{H_f}$.
}
\end{definition}
For any anti-holomorphic topological automorphism on a Hopf surface $H_f$, we have a covering transformation $\widehat{\sigma}^2$ on $\mathbb{C}^2-\{(0,0)\}$ but covering transformation group is generated by the contraction, which implies $\widehat{\sigma}^2=f^n$ for some positive integer $n$.
\begin{definition}\rm{
An anti-holomorphic involution $\sigma: H_f\rightarrow H_f$ is called an even (respectively, odd) real structure if $\widehat{\sigma}^2$ corresponds to even (respectively, odd) power of $f$.
}
\end{definition}
Let the contraction $f:W\rightarrow W$ be 
$$
f(z_1,z_2)=\left(\frac{1}{2}z_1,\frac{-1}{2}z_2\right)
$$
and $\sigma_f$ be even real structure on Hopf surface $H_f$, which is equivalent to standard real structure $s_f$ on $H_f$ whose lift on $W=\mathbb{C}^2-\{(0,0)\}$ is standard complex conjugation. The fixed points set of $H_f$ under $\sigma_f$ is topologically homeomorphic to 
$$
\bigslant{\mathbb{R}^2-\{(0,0)\}}{(x,y)\sim \left(\frac{x}{2},-\frac{y}{2}\right)}
$$ 
which is further homeomorphic to Klein bottle.

Hence, $\bigslant{H_f}{\langle\sigma_f\rangle}$ is a non-orientable topological 4-manifold whenever the determinant of the real part of contraction is negative. For more details on the classification of real primary Hopf surfaces, see \cite{ZK}.

Tha complex structure on Hopf surface $H_f$ induces a $d$-complex structure on $\bigslant{H_f}{\langle\sigma_f\rangle}$ 
(see Theorem \ref{dcomphopf}).

 \begin{theorem}\label{dcomphopf}
 Let $(H_f,\dholos{H}_f)$ be a Hopf surface $H_f$ with complex structure $\dholos{H}_f$ and let $G=\{1,\sigma_{H_f}\}$ be the group generated by even real structure $\sigma_{H_f}$ on $H_f$. The quotient space $Y=\bigslant{H_f}{G}$ has a $d$-complex structure $\mathfrak{Y}$.
 \end{theorem}
 \begin{proof}
 The group $G$ acts properly discontinuous on Hopf surface which implies $\bigslant{X}{G}$ will be a Hausdorff space \cite[Chapter 1, Proposition 1.1.2]{EBJJEJMGGG}
 
 Let $x\in X$ such that
 \begin{enumerate}[label=\textbf{Case:\arabic*.}]
 \item $x\neq \sigma_{H_f}(x)$
 
 Using \cite[Proposition 1.1.1]{EBJJEJMGGG}, we can have chart $(U_x,\phi_x)$ around the point $x\in X$ such that $U_x\cap \sigma_X(U_x)=\phi$ and $\phi_x(x)=0$. Take $V_x=\{[x]\in \bigslant{X}{G}:x\in U_x\}$ and $\phi'_x([x])=\phi_x(x)$. Then $\phi'_x(V_x)$ is homeomorphic to some open subset in $\mathbb{C}^2$ with $\phi'_x([x])=0$. Also, note that $\phi'_x$ is holomorphic as $\phi_x$ is holomorphic.
 \item $x=\sigma_{H_f}(x)$
 
 Let $(U_x,\phi_x)$ be analytic chart around $x\in H_f$ such that $\phi_x(x)=0$. By possibly shrinkig and using \cite[Theorem 3.6]{ZK}, there is a holomorphic automorphism $g:H_f\rightarrow H_f$ such that $g\circ \sigma_{H_f}\circ g^{-1}:U_x\rightarrow U_x$ is standard real structure. Furthermore, the map $\phi_x\circ g\circ \sigma\circ g^{-1}\circ \phi_x^{-1}:V_x\rightarrow V_x$ is complex conjugation for some open subset $V_x\subset \mathbb{C}^2$ homeomorphic to $U_x$.
 \[\begin{tikzcd}[ampersand replacement=\&]
	{U_x} \& {U_x} \& {V_x} \\
	{U_x} \& {U_x} \& {V_x}
	\arrow["g", from=1-1, to=1-2]
	\arrow["\sigma"', from=1-1, to=2-1]
	\arrow["{\phi_x}", from=1-2, to=1-3]
	\arrow[from=1-3, to=2-3]
	\arrow["g"', from=2-1, to=2-2]
	\arrow["{\phi_x}"', from=2-2, to=2-3]
\end{tikzcd}\]
Let $\rho:\mathbb{C}^2\rightarrow \mathbb{C}^+\times \mathbb{C}$ be the map given by
\[
    \rho(a+bi,a_1+b_1i) = \left\{\begin{array}{lr}
        (a+bi,a_1+b_1i), & \text{if } b\geq 0\\
        (a-bi,a_1-b_1i), & \text{if } b<0
        \end{array}\right.
  \]
then the map $\rho\circ \phi_x\circ g:U_x\rightarrow \mathbb{C}^+\times \mathbb{C}$ is $\sigma$ invariant. Which further implies the map $\phi'_x=\rho\circ \phi_x^U\circ g:\bigslant{U_x}{G}\rightarrow \mathbb{C}^+\times \mathbb{C}$ is well defined.
\end{enumerate}
Hence, we have an atlas $\mathcal{V}=\{(V_x,\phi'_x)\}_{[x]\in \bigslant{X}{G}}$ on $Y=\bigslant{X}{G}$, which is a $d$-holomorphic atlas as described below.

Take two charts $(V_x,\phi'_x)$ and $(V_{x_1},\phi'_{x_1})$ around the points $[x]\neq [x_1]$ such that $[y]\in V_x\cap V_{x_1}\neq \phi$, then
\begin{enumerate}[label=\textbf{Case:\arabic*.}]
 \item if $x$ as wells as $x_1$ are not $\sigma_X$ fixed points, then we have
 $$
 g_x\circ g_{x_1}^{-1}([y])=\phi_x\circ \phi_{x_1}^{-1}(y)
 $$
 which is holomorphic;
 \item if any one of, say $x$ is a $\sigma_X$ fixed point, then we have
 $$
 g_x\circ g_{x_1}^{-1}([y])=\phi\circ h_x\circ g_x'\circ \phi_x^{-1}(y)
 $$
 which is either holomorphic or anti-holomorphic on each connected component of $V_x\cap V_{x_1}$.
 \item if both $x$ as wells as $x_1$ are $\sigma_X$ fixed points, we have
 $$
 \phi'_x\circ {\phi'_{x_1}}^{-1}([y])=\rho\circ \phi_x\circ g\circ g^{-1}\circ \phi_{x_1}^{-1}(y)
 $$
 which is holomorphic or anti-holomorphic on each connected component of $V_x\cap V_{x_1}$.
 \end{enumerate} 
Hence the atlas $\{(V_x,\phi'_x\}_{[x]\in Y}$ is a $d$-holomorphic atlas.
 \end{proof}

 The \emph{Hodge number} $h^{(p,q)}(Y)$ is defined as
 $$
 h^{(p,q)}(Y)=\text{dim}_\mathbb{R} H^{(p,q)}_{dR}(Y)\quad(1\leq p,q\leq 2)
 $$
 From \cite[Theorem 4.5]{SW1}, $H^{(p,q)}_{dR}(X)\simeq \mathbb{C}\otimes_\mathbb{R}H^{(p,q)}_{dR}(Y)$, which implies
 $$
 h^{(p,q)}(Y)=\text{dim}_{\mathbb{R}}H^{(p,q)}_{dR}(Y)=\text{dim}_{\mathbb{C}}\mathbb{C}\otimes H^{(p,q)}_{dR}(Y)=\text{dim}_{\mathbb{C}}H^{(p,q)}_{dR}(X)=h^{(p,q)}(X)\quad(1\leq p,q\leq 2)
 $$
 
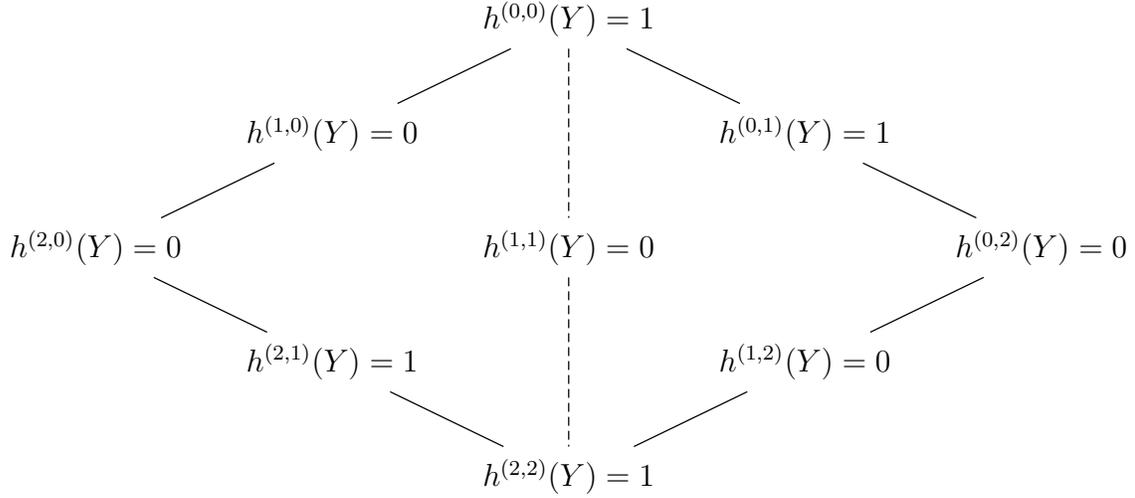
\[\begin{tikzcd}[ampersand replacement=\&,column sep=small]
	\&\& {h^{(0,0)}(Y)=1} \\
	\& {h^{(1,0)}(Y)=0} \&\& {h^{(0,1)}(Y)=1} \\
	{h^{(2,0)}(Y)=0} \&\& {h^{(1,1)}(Y)=0} \&\& {h^{(0,2)}(Y)=0} \\
	\& {h^{(2,1)}(Y)=1} \&\& {h^{(1,2)}(Y)=0} \\
	\&\& {h^{(2,2)}(Y)=1}
	\arrow[no head, from=1-3, to=2-2]
	\arrow[no head, from=1-3, to=2-4]
	\arrow[dashed, no head, from=1-3, to=3-3]
	\arrow[no head, from=2-2, to=3-1]
	\arrow[no head, from=2-4, to=3-5]
	\arrow[no head, from=3-1, to=4-2]
	\arrow[dashed, no head, from=3-3, to=5-3]
	\arrow[no head, from=3-5, to=4-4]
	\arrow[no head, from=4-2, to=5-3]
	\arrow[no head, from=4-4, to=5-3]
\end{tikzcd}\]
 
Since Hodge numbers are not symmetric about the vertical dotted line, the quotient space $Y=\bigslant{H_f}{\langle \sigma_f \rangle}$ is an example of non-orientable, non $d$-K\"{a}hler manifold; where $H_f=\bigslant{\mathbb{C}^2-\{(0,0)\}}{\langle f \rangle}$ is a Hopf surface, $f:\mathbb{C}^2\rightarrow \mathbb{C}^2$ is contraction given by $f(z_1,z_2)=(\frac{1}{2}z_1,\frac{-1}{2}z_2)$ and $\sigma_f$ is even real structure on $H_f$.

\section{$d$-holomorphic connections and Chern classes}\label{section5}
The notion of $d$-holomorphic connection on $d$-holomorphic vector bundles on Klein surfaces ($d$-complex manifold of complex dimension $1$) is studied in \cite{SAAJ}. In this section, we prove that if a $d$-holomorphic bundle on a compact $d$-K\"{a}hler manifold admits a $d$-holomorphic connection, then all of its Chern classes will vanish (see Theorem \ref{maintheorem}). 

For two $d$-holomorphic bundles $\mathcal{E},\mathcal{F}$ on a compact $d$-complex manifold $(X,\dholos{X})$ and a first order differential operator $P\in Hom_{\mathcal{O}_X^{dh}}(E,F)$ there is a map $\sigma_1(P):{T^\star}^{(1,0)}\rightarrow Hom_{\mathcal{O}_X^{dh}}(E,F)$. For $\nu=\{\nu_{U_i}\}_{i\in I}\in T^{(1,0)}(X)$ and $s=\{s_{U_i}\}_{i\in I}\in E(X)$ along with $f=\{f_{U_i}\}_{i\in I}\in \mathcal{O}_X^{dh}(X)$ such that $d(f_{U_i})=\nu_{U_i}$ $(i\in I)$, the map $\sigma_1(P)$ is given by $\sigma_1(P)_{U_i}(d\nu_{U_i})(s_{U_i})=P(f_{U_i}-f_{U_i}(x))(s_{U_i})$ in a chart ${(U_i,\phi_i)}$ around a point $x\in X$, known as the symbol map. There is an exact sequence known as the symbol exact sequence
$$
0\rightarrow \mathcal{H}om_{\mathcal{O}_X^{dh}}(\mathcal{E},\mathcal{F})\hookrightarrow \mathcal{D}iff_1(\mathcal{E},\mathcal{F})\xrightarrow{\sigma_1}T^{(1,0)}\otimes \mathcal{H}om_{\mathcal{O}_X^{dh}}(\mathcal{E},\mathcal{F}). 
$$
The symbol exact sequence further induces Atiyah exact sequence for a $d$-holomorphic bundle $E$ on a compact $d$-complex manifold 
$$
0\rightarrow \mathcal{E}nd_{\mathcal{O}_X^{dh}}(\mathcal{E})\rightarrow \mathcal{A}t_d(\mathcal{E})\rightarrow T^{(1,0)}\rightarrow 0
$$
where $\mathcal{A}t_d(\mathcal{E})(U)$ is the collection of those first order differential operators whose symbol map is global sections of $T^{(1,0)}$ considering $T^{(1,0)}$ as an $\mathcal{O}_X^{dh}$-submodule of $T^{(1,0)}\otimes \mathcal{E}nd_{\mathcal{O}_X^{dh}}(\mathcal{E})$.
 
The extension class associated with the Atiyah exact sequence for $d$-holomorphic bundle on a compact $d$-complex manifold is called the Atiyah class. Following the same line of arguments as in \cite[Proposition 21]{SAAJ}, it can be proved that the curvature $(1,1)$ form associated to a smooth $d$-complex connection compatible with $d$-holomorphic structure of the bundle is negative of the Atiyah class of the $d$-holomorphic bundle. 

\begin{lemma}\label{lemma5}
Let $X$ be a compact $d$-K\"{a}hler manifold and $\omega\in \mathcal{A}^{(p,q)}(X)$ be a $d_\mathcal{C}$-closed form of type $(p,q)$. If $\omega$ is $\overline{\partial}$-exact, there exist a $d$-complex $(p-1,q-1)$ form $\phi\in \mathcal{A}^{(p-1,q-1)}(X)$ such that $\omega=\partial \overline{\partial}\phi$.
\end{lemma}
\begin{proof}
Since $\omega$ is $d_\mathcal{C}$-closed and $\overline{\partial}$-exact, implies $\omega$ is $\partial$ and $\overline{\partial}$-closed. Let $\omega=\overline{\partial}(\eta)$ for some $\eta\in \mathcal{A}^{(p,q-1)}(X)$. 

Using Hodge theory for $d$-complex manifold, we have a decomposition
$$
\eta=\alpha+\partial \beta+\partial^\star \gamma,
$$
with $\alpha\in \mathcal{H}^{(p,q-1)}_\partial(X)$, $\beta\in \mathcal{A}^{(p-1,q-1)}(X)$ and $\gamma\in \mathcal{A}^{(p+1,q-1)}(X)$. Since $X$ is $d$-K\"{a}hler, we have $\mathcal{H}_\partial^{(p,q)}(X)=\mathcal{H}_{\overline{\partial}}^{(p,q)}(X)$, which implies $\alpha\in \mathcal{H}_{\overline{\partial}}^{(p,q-1)}(X)\simeq H_{dB}^{(p,q-1)}(X)$ i.e. $\alpha$ is $\overline{\partial}$-closed. We have 
\begin{align*}
\omega&=\overline{\partial}\eta=\overline{\partial}\alpha+\overline{\partial}\partial \beta+\overline{\partial}\partial^\star \gamma\\
&=\overline{\partial}\partial\beta+\overline{\partial}\partial^\star \gamma.
\end{align*}
Using $\overline{\partial}\partial^\star=-\partial^\star \overline{\partial}$ (follows from \ref{4.3}) and $\overline{\partial}\partial=-\partial\overline{\partial}$, we have
$$
\omega=-\partial\overline{\partial}\beta-\partial^\star\overline{\partial}\gamma.
$$
Since, $\partial(\omega)=0$ and $\partial \partial \overline{\partial}(\beta)=0$, we have $\partial \partial^\star \overline{\partial}\gamma=0$. Using corollary \ref{vanish} we have $\partial^\star \overline{\partial}\gamma=0$. 

Furthermore, $\omega=-\partial\overline{\partial}\beta\in Im(\partial\overline{\partial})$. So, $Im(\overline{\partial})\cap Ker(\partial)\subset Im(\partial\overline{\partial})$.

It is easy to prove that $Im(\partial\overline{\partial})\subset Ker(\partial)\cap Im(\overline{\partial})$. Hence, we have $$
Im(\overline{\partial})\cap Ker(\partial)= Im(\partial\overline{\partial})
$$
\end{proof}

The following result is a generalization of \cite[Section 3, Theorem 4]{MA} (see also \cite[Theorem 3.6]{IBNR}).
We refer to \cite{SW1} for the definition of Chern classes for $d$-complex vector bundles over a $d$-complex manifolds.

\begin{theorem}\label{maintheorem}
Let $\mathcal{E}$ be a $d$-holomorphic vector bundle on a compact $d$-K\"{a}hler manifold $X$. Suppose the bundle $\mathcal{E}$ admits a $d$-holomorphic connection  i.e. $at_d(\mathcal{E})=0$, then all the chern classes $c_{2i}\in H^{4i}(X,\mathbb{R})$ $\big(\text{respectively, }c_{2i+1}\in H^{4i+2}(X,\mathcal{L})\big)$ of $\mathcal{E}$ are zero.
\end{theorem}
\begin{proof}
Since $\mathcal{E}$ admits a $d$-holomorhic connection, it follows that the Atiyah class $at_d(\mathcal{E})=0$. This
further implies that the Dolbeault cohomology class associated to the curvature, i.e. 
$[R^{(1,1)}]\in H^{(1,1)}(X,\mathcal{E}nd_{\mathcal{O}_X^{dh}}(\mathcal{E}))$ vanishes. In other words, there is a 
$\overline{\partial}$-closed form 
$\mathfrak{r}\in \mathcal{A}^{(1,0)}\big(\mathcal{E}nd_{\mathcal{O}_X^{dh}}(\mathcal{E})\big)(X)$ such that 
$\overline{\partial}_E(\mathfrak{r})=R^{(1,1)}$, where $\overline{\partial}_E$ is the Cauchy-Riemann operator 
for the bundle associated to the sheaf $\mathcal{E}nd_{\mathcal{O}_X^{dh}}(\mathcal{E})$ (see \cite[Definition 8]{SAAJ}). 

Let $(U_\alpha,\phi_\alpha)$ be a chart on $d$-complex manifold $(X,\dholos{X})$, $(s_{\alpha i})_{i=1}^n$ be fixed frame for $E_{U_\alpha}$ and $\Omega_\alpha$ be curvature form for the unique unitary connection compatible with $d$-holomorphic structure of $E$.  Then, we have
$$
\mathfrak{r}_\alpha=\mathfrak{r}\big|_{U_\alpha}(s_{\alpha j})=\displaystyle\sum_{i=1}^n\mathfrak{r}_{\alpha ij}\otimes s_{\alpha i}\quad \text{where, }\mathfrak{r}_{\alpha ij}\in \mathcal{A}^{(1,0)}(U_\alpha)
$$
with the compatibility condition,
\[
\mathfrak{r}_\beta(x)=\left\{\begin{array}{lr}
        \left(\cocycle{\alpha}{\beta}{E}\right)^{-1}(x)\mathfrak{r}_\alpha(x)\cocycle{\alpha}{\beta}{E}(x), & \text{ if }\trans{\alpha}{\beta}\text{ is holomorphic}\vspace{7pt}\\
        \left(\cocycle{\alpha}{\beta}{E}\right)^{-1}(x)\overline{\mathfrak{r}}_\alpha(x)\overline{\cocycle{\alpha}{\beta}{E}}(x), & \text{ if }\trans{\alpha}{\beta}\text{ is anti-holomorphic}
        \end{array}\right.\quad (\text{for }x\in U_\alpha\cap U_\beta)
\]
The equality $R=\overline{\partial}_E(\mathfrak{r})$ implies that $\Omega_\alpha=\overline{\partial}\mathfrak{r}_\alpha$.

Consider a $Gl_n(\mathbb{C})$-invariant homogenous polynomial $f:\mathfrak{gl}_n(\mathbb{C})\rightarrow \mathbb{C}$ of degree $p$. Then $f(R)$ can be expressed as a family $\{f(\Omega_\alpha)\}_{\alpha\in I}$ with respect to some $d$-holomorphic atlas $\mathcal{U}=\{(U_\alpha,\phi_\alpha)\}_{\alpha\in I}$ with the compatibility condition,
\[
f(\Omega_\beta)=\left\{\begin{array}{lr}
        f(\Omega_\alpha), & \text{ if }\trans{\alpha}{\beta}\text{ is holomorphic}\vspace{7pt}\\
        \overline{f(\Omega_\alpha)}, & \text{ if }\trans{\alpha}{\beta}\text{ is anti-holomorphic}
        \end{array}\right.
\]
Also, viewing $f$ as a multilinear function $f:\mathfrak{gl}_n(\mathbb{C})\times\dots\times\mathfrak{gl}_n(\mathbb{C})(p \text{ times})\rightarrow \mathbb{C}$, we have 
$$
f(\Omega_\alpha,\dots,\Omega_\alpha)=\overline{\partial}f(\mathfrak{r}_\alpha,\overline{\partial}\mathfrak{r}_\alpha,\dots,\overline{\partial}\mathfrak{r}_\alpha)
$$
This implies that $f(R)\in \mathcal{A}^{(p,p)}(X)$ is $\overline{\partial}$-exact. On the other hand from the 
Chern-Weil theory we know that $f(R)$ is $d_\mathcal{C}$-closed. Using the decomposition and comparing bi-degrees 
for $d_\mathcal{C}(f(R))\in \mathcal{A}_\mathcal{C}^{2p+1}(X)$, we can say that $f(R)$ is $\partial$-closed. 
Considering the Hodge decomposition for $d$-K\"{a}hler manifold and using Lemma \ref{lemma5}, we have 
$f(R)\in Im(\partial\overline{\partial})$, hence $f(R)$ is $d_\mathcal{C}$-exact. Since $f(R)$ is $d_\mathcal{C}$-closed, 
it represents an element in the de-Rham cohomology group but $f(R)$ is $d_\mathcal{C}$-exact as well, which further 
implies that the de-Rham cohomology class $[f(R)]\in H^{2p}(X,\mathcal{C})$ vanishes. 

Applying the above procedure for $i^{\mathrm{th}}$ elementary $GL_n(\mathbb{C})$-invariant function of the matrix $X=R|_{U_\alpha}$, we get $c_i(E)=0$ for all $i\in I$. 
\end{proof}

\appendix 
\section{Sobolev spaces}\label{appendix:sectionA}
\renewcommand{\thesubsection}{\textbf{A.\arabic{subsection}}}
This section presents the analysis of differential operators between $d$-complex vector bundles on 
$d$-complex manifolds. The exposition is adapted from \cite[Chapter 4]{ROW}, with necessary modifications 
tailored to the context of $d$-complex manifolds. We refrain from giving routine details for the proofs, 
as they follow a similar line of reasoning as in \cite[Chapter 4]{ROW}.

\subsection{Sobolev norm on $d$-complex bundle}\label{appendix1}

Let $f\equiv(f_1,\dots,f_m):\mathbb{R}^n\rightarrow \mathbb{C}$ be a compactly supported function and for some $s\in \mathbb{Z}$, define Sobolev norm,
$$
||f||^2_{s,\mathbb{R}^n}=\displaystyle\sum_{i=1}^m\displaystyle\int_{\mathbb{R}^n} |\widehat{f_i}(y)|^2(1+|y|^2)^s dy,
$$
where $\widehat{f_i}$ is Fourier transform of $f_i$ ($i=1,\dots,m$) given by,
$$
\widehat{f_i}(y)=(2\pi)^{-n}\int e^{-i\langle x,y\rangle} f_i(x)dx
$$

Let $(X,\dholos{X})$ be a $d$-complex manifold with $d$-holomorphic atlas $\{(U_i,\phi_i)\}$ and smooth partition of unity $\{\rho_i,i\in I\}$ sub-ordinate to the cover $\{U_i:i\in I\}$. Let $f=\{f_{U_i\cap U}\}_{i\in I}\in \mathcal{A}_d^0(U)$ be a smooth $d$-complex function, define the Sobolev norm
$$
||f||_{s,\mathcal{C}}=\displaystyle\sum_{i\in I} ||\rho_i f_{U_i\cap U}\circ \phi_i^{-1}||_{s,\mathbb{R}^n}
$$

In case of multivalued smooth function, taking Sobolev norm of each component function and Euclidean norm the vector will give the Sobolev norm of the multivalued smooth function. 

Using above procedure and smooth partition of unity subordinate to the covering $\{U_i\}_{i\in I}$, we have the Sobolev norm $||u||_{s,E}$ of a global section $u\in E(X)$ of a smooth $d$-complex bundle.

For an $n$-tupe $\alpha=(\alpha_1,\alpha_2,\dots,\alpha_n)\in \mathbb{Z}^n$ and a smooth function $f:\mathbb{R}^n\rightarrow \mathbb{C}^n$, let $|\alpha|=(\alpha_1+\alpha_2+\dots+\alpha_n)$, $y^\alpha=y_1^{\alpha_1}y_2^{\alpha_2}\dots y_n^{\alpha_n}$ and $D^\alpha\equiv (-i)^{|\alpha|} D^{\alpha_1}_1\circ D^{\alpha_2}_2\circ\dots\circ D^{\alpha_n}_n$, where $D_i\equiv \frac{\partial}{\partial x_i}$, we have a property of fourier transform as follows.
\begin{lemma}\cite[Lemma 1.7.1]{LH}\label{diff1}
Let $f:\mathbb{R}^n\rightarrow \mathbb{C}$ be a compactly supported smooth function and $\alpha=(\alpha_1,\dots,\alpha_n)\in \mathbb{Z}^n$. Then
$$
\widehat{D^\alpha f}(y)=y^\alpha\widehat{f}(y)
$$
\end{lemma}
The above lemma implies that if the Sobolev norm $||f||_{s,\mathbb{R}^n}$ is bounded, then $f$ can be assumed differentiable about $s$ times.

Define $W^s_\mathcal{C}(E)(X)$ to be completion of $\mathcal{A}_\mathcal{C}^0(E)(X)$ with respect to the norm $||.||_{s,E}$. By Lemma \ref{diff1}, the space $W^s_\mathcal{C}(E)(X)$ can be thought of as a Banach space approximating space of $s$-times differentiable global sections of the bundle $E$. The space $W^s_\mathcal{C}(E)(X)$ inherits the inner product structure of $\mathcal{A}_\mathcal{C}^0(E)(X)$ (as described above), and thus we have a sequence of Hilbert spaces, 
$$
\dots\supset W^s_\mathcal{C}(E)(X)\supset W^{s+1}_\mathcal{C}(E)(X)\supset \dots \supset W^{s+j}_\mathcal{C}(E)(X)\supset\dots
$$
\begin{proposition}[Sobolev]\cite[Chapter 4, Section 1, Proposition 1.1]{ROW}\label{Sobolev}
Let $\text{dim}_{\mathbb{R}}X=2n$ and suppose that $s=\left[n \right]+k+1$. Then
$$
W^s_\mathcal{C}(E)(X)\subset \mathcal{A}^0_{\mathcal{C},k}(E)(X),
$$
where $\mathcal{A}^0_{\mathcal{C},k}(E)(X)$ is the space of $k$-times differentiable global sections of the bundle $E$.
\end{proposition}
\begin{proposition}[Rellich]\cite[Chaper 4, Section 1, Proposition 1.2]{ROW}\label{Relich}
The canonical operator $j:W^t_\mathcal{C}(E)(X)\hookrightarrow W^s_\mathcal{C}(E)(X)$ , where $s<t$ is a compact operator.
\end{proposition}
\begin{definition}\label{diffOP}\rm{
Define a class of operators $\mathrm{OP_m(E, F)}\subset \mathrm{Hom}_\mathcal{C}(E,F)$ $(k\in \mathbb{N})$ such that if $T=\{T_{U_i}\}_{i\in I}\in \mathrm{OP_m(E, F)}$ with respect to some atlas $\mathcal{U}=\{(U_i,\phi_i)\}_{i\in I}\in \dholos{X}$ then for all $s\in \mathbb{Z}$ there is a continuous extension of each $T_{U_i}(i\in I)$,
$$
T_{U_i,s}:W^s_\mathcal{C}(E)(U_i)\rightarrow W^{s-m}_\mathcal{C}(E)(U_i)
$$
We call the elements of $\mathrm{OP_m(E, F)}$, operators of order $m$ from $E$ to $F$.
}
\end{definition}

Following the arguments as in \cite[Chapter 4, Section 2, Proposition 2.2]{ROW}, we have $\mathrm{Diff_m(E, F)}\subset \mathrm{OP_m(E, F)}$.

\subsection{Pseudo-differential operator on Euclidean space}\label{appendix2}

In this section, we recall some basic concepts pertaining to pseudo-differential operators on the space of smooth functions on open subsets of $\mathbb{R}^n$, with compact support. For details see \cite[Chapter 4, Section 3]{ROW}.

Let $U\subset \mathbb{R}^n$ (an open subset) and $p(x,y)$ be a polynomial of degree $m$ in $y\in \mathbb{R}^n$ with coefficients as smooth functions of $x\in U$. We can associate to the polynomial $p(x,y)$, a differential operator $P=p(x,D)$ by replacing monomials $y^\alpha=y_1^{\alpha_1}\dots y_n^{\alpha_n}$, $\alpha=(\alpha_1,\dots,\alpha_n)\in \mathbb{N}^n$ with the operator $D^\alpha=(-i)^{|\alpha|}D_1^{\alpha_1}\circ\dots\circ D_n^{\alpha_n}$, where $D_i\equiv \frac{\partial}{\partial x_i}$ and $|\alpha|=\alpha_1+\alpha_2+\dots+\alpha_n$. The polynomial $p(x,y)$ operates on a given function $f$ with domain $U$ as,
$$
p(x,y).f(x)=p(x,D)(f)(x)=\int_{\mathbb{R}^n} p(x,y)\widehat{f}(y)e^{i\langle x,y \rangle} dy\quad(\text{using Lemma \ref{diff1}}).
$$

Following \cite[Chapter 4, Section 3]{ROW}, we can generalize the notion of $m$-th order differential operator by replacing it with the polynomial having some suitable asymptotic properties as described below.

\begin{definition}\rm{
Let $U\subset \mathbb{R}^n$ be an open subset and let $m$ be an integer.
\begin{enumerate}
\item Define $\tilde{S}^m(U)$ to be the class of functions $p(x,y)$ defined on $U\times \mathbb{R}^n$ with the property that for any compact set $U_k\subset U$, for any multi indices $\alpha,\beta$, there exist a constant $C_{\alpha\beta k}$ depending only on $\alpha,\beta,k$ and polynomial $p$ such that,
$$
|D_x^\alpha D_y^\beta p(x,y)|\leq C_{\alpha\beta k}(1+|y|)^{m-|\alpha|}
$$
for all $x\in U_k$ and $y\in \mathbb{R}^n$.
\item Let $S^m(U)\subset \tilde{S}^m(U)$ be collection of those polynomials $p(x,y)$ for which the limit $\sigma_m(p)(x,y)=\lim_{\lambda\rightarrow\infty}\frac{p(x,\lambda y)}{\lambda^m}$ exist for $y\in \mathbb{R}^n$ (non-zero) and for a non-zero smooth function $\psi\in C^\infty(\mathbb{R}^n)$ such that $\psi\equiv 0$ in a neighbourhood of $y=0$ and $\psi\equiv 1$ outside a unit ball centered at $y=0$, we have
$$
p(x,y)-\psi(y)\sigma_m(p)(x,y)\in \tilde{S}^{m-1}(U)
$$
\item Define $\tilde{S}_0^m(U)\subset \tilde{S}^m(U)$ be the class of functions $p(x,y)$ having compact support as a function of $x$ in $U_k\subset U$.
\item Also, define $S^m_0(U)\subset S^m(U)$ be the class of functions $p(x,y)\in S^m(U)$ having compact support as a function of $x$ in $U_k\subset U$ i.e. $S^m_0(U)=S^m(U)\cap \tilde{S}^m_0(U)$.
\end{enumerate}
}
\end{definition}
\begin{remark}\rm{
If polynomial $p(x,y)$ is of degree $m$ in $y\in \mathbb{R}^n$ then $p(x,y)\in S^m(U)$. Furthermore, if coefficients of $y^\alpha$ ($|\alpha|=m$) in $p(x,y)$, which are smooth functions of $x\in U$, have compact support then $p\in S^m_0(U)$.
}
\end{remark}
\begin{lemma}\cite[Chapter 4, Section 3, Lemma 3.2]{ROW}
For $p\in S^m(U)$, $\sigma_m(p)$ is a smooth function on $U\times (\mathbb{R}^n\setminus\{0\})$, homogenous of degree $m$ in $y$.
\end{lemma}
\begin{definition}\label{pseuddef}\rm{
For any $p\in \tilde{S}^m(U)$ and $f\in C_c^\infty(U)\subset C^\infty(U)$ having compact support, define 
$$
L(p)(f)(x)=\displaystyle\int_{\mathbb{R}^n} p(x,y)\widehat{f}(y)e^{i\langle x,y \rangle}dy
$$
where, $\widehat{f}(y)$ is Fourier transform of $f$ and $\langle x,y \rangle$ is euclidean inner product of vectors $x,y\in \mathbb{R}^n$. The operator $L(p)$ is called a pseudo-differential operator canonical to the polynomial $p\in \tilde{S}^m(U)$ of order $m$.
}
\end{definition}
From \cite[Chapter 4, Section 3, Lemma 3.3]{ROW} and \cite[Chapter 4, Section 3, Theorem 3.4]{ROW}, it follows that for $p\in \tilde{S}^m(U)$, $L(p)$ maps $C^\infty_c(U)$ into $C^\infty(U)$ and for $p\in \tilde{S}^m_0(U)$, $L(p)$ is an operator of order $m$.

\subsection{Parametrix for elliptic differential operator on $d$-complex bundle}\label{appendix3}
Pseudo-differential operators for $d$-complex vector bundles on a $d$-complex manifold has been described in Subsection \ref{subsection3.1} and reason for defining it, is to obtain a class of operators, which include the inverses.
 
In this section, we will discuss about parametrix $\widetilde{L}$ of an operator $L\in \mathrm{PDiff_m(E, F)}$ between $d$-complex vector bundles $E$ and $F$ on a $d$-complex manifold $(X,\dholos{X})$ (for details see Subsection \ref{subsection3.1}), assuming symbol of $L$ has some property as discussed below, fundamental decomposition theorem for elliptic self-adjoint differential operators, which is useful to get the orthogonal decomposition for a given elliptic complex, which further helps in getting Hodge decomposition for de-Rham chain complex \eqref{derhamcomp} as well as Dolbeault chain complex \eqref{dolbeaultcomp}.

\begin{definition}\rm{
For $d$-holomorphic bundles $E$, $F$ on a $d$-complex manifold $(X,\dholos{X})$ with $d$-holomorphic atlas $\mathcal{U}=\{(U_i,\phi_i\}_{i\in I}$, an element $s=\{s_{U_i}\}_{i\in I} \in \mathrm{Smbl}_m(E,F)$ (see Subsection \ref{subsection3.1}) is called elliptic, if $s_{U_i}\in \mathrm{Smbl}_m(E|_{U_i},F|_{U_i})$ is elliptic for each $i\in I$ and an operator $L\in \mathrm{PDiff_m(E, F)}$ is called an elliptic if its symbol $\sigma_m(L)$ is elliptic. For more details, see \cite[Chapter 4, Section 4]{ROW}.
}
\end{definition}
\begin{definition}\rm{
An element $\widetilde{L}\in \mathrm{PDiff}_{-m}(E,F)$ is called parametrix for an operator $L\in \mathrm{PDiff}_{m}(E,F)$ if,
\begin{align*}
L\circ \widetilde{L}-I_F&\in \mathrm{OP}_{-1}(F)\\
\widetilde{L}\circ L-I_E&\in \mathrm{OP}_{-1}(E)
\end{align*}
}
\end{definition}
Using Theorem \ref{symbl}, properties of symbol operators (see Theorem \ref{symblprop}) and following in the same of arguments as in \cite[Chapter 4, Section 4, Theorem 4.4]{ROW}, there exists a parametrix for a given elliptic operator $L\in \mathrm{PDiff_m(E, F)}$. 

Note that, an operator $L\in \mathrm{OP_m(E, F)}$ is compact if for each $s\in \mathbb{Z}$, the extension $L_s:W^s_\mathcal{C}(E)(X)\rightarrow W^{s-m}_\mathcal{C}(F)(X)$ is compact operator as a mapping between Banach spaces.
\begin{proposition}\label{prop2}
For $d$-complex vector bundle $E$ on a $d$-complex manifold $(X,\dholos{X})$, let $S\in \mathrm{OP}_{-1}(E,E)$. Then, $S$ is a compact operator of order $0$.
\end{proposition}
\begin{proof}
The proof follows in same line of arguements as in \cite[Chapter 4, Section 4, Proposition 4.5]{ROW} using Proposition \ref{Relich} (ch.\cite[Chapter 4, Section 1, Proposition 1.4]{ROW}).
\end{proof}

Let $E$ and $F$ are fixed $d$-Hermitian vector bundles on $d$-complex manifold $(X,\dholos{X})$ with line bundle $\mathcal{L}$ as described in Subsection \ref{subsec2.1}. The $d$-Hermitian inner product on $E$ and $F$, induces an inner product on the spaces $\mathcal{A}^0_\mathcal{C}(E)(X)$ and $\mathcal{A}^0_\mathcal{C}(F)(X)$, as described in Subsection \ref{subsecad} with their completions $W^0_\mathcal{C}(E)(X)$ and $W^0_\mathcal{C}(F)(X)$, respectively, under the $L^2$ norm.

For $L\in \mathrm{Diff_m(E, F)}$, define the space
$$
\mathcal{H}_L(X)=\{\alpha\in \mathcal{A}^0_\mathcal{C}(E)(X):L(\alpha)=0\}
$$
with orthogonal complement $\mathcal{H}_L^\perp(X)$ in the Hilbert space $W^0_\mathcal{C}(E)(X)$, which is a closed subspace of $W^0_\mathcal{C}(E)(X)$. Moreover, if $L$ is an elliptic operator, then $\mathcal{H}_L(X)$ is finite-dimensional; hence closed in $W^0_\mathcal{C}(E)(X)$ as well (see Theorem \ref{findim}).
\begin{proposition}\cite[Chapter 4, Section 4, Proposition 4.6]{ROW}\label{prop3}
Let $B$ be a Banach space and let $S: B\rightarrow B$ be a compact operator. Let $T=I-S$. Then
\begin{enumerate}
\item ker($T$) is finite dimensional
\item $T(B)$ is closed in $B$ and Coker($T$) is finite dimensional.
\end{enumerate}
\end{proposition}
\begin{proof}
The proof depends on the fundamental finiteness criterion in functional analysis, which asserts that a locally compact topological vector space is a finite dimension. For more details see, \cite{WR} and \cite[Chapter 4, Section 4 Proposition 4.6]{ROW}.
\end{proof}
\begin{definition}\rm{
A Fredholm operator, $T$ on a Banach space is an operator with a finite-dimensional kernel and finite-dimensional cokernel.
}
\end{definition}

Using existence of parametrix, Propositions \ref{prop2} and \ref{prop3}, we have
\begin{theorem}\label{theorem1}
Let $L\in \mathrm{PDiff_m(E, F)}$ be an elliptic pseudo-differential operator. There exists a parametrix $P$ for $L$ such that $L\circ P$ (respectively, $P\circ L$) has continuous extension as a Fredholm operator $(L\circ P)_s:W^s_\mathcal{C}(F)(X)\rightarrow W^s_\mathcal{C}(F)(X)$ (respectively, $(P\circ L)_s:W^s_\mathcal{C}(E)(X)\rightarrow W^s_\mathcal{C}(E)(X)$) for each integer $s$. 
\end{theorem}
\begin{theorem}\label{theorem2}
For a given elliptic differential operator $L\in \mathrm{Diff_m(E, F)}$, let $\alpha\in W^s_\mathcal{C}(E)(X)$ has propery that $L_s(\alpha)=\beta\in \mathcal{A}^0_\mathcal{C}(F)(X)$. Then $\alpha\in \mathcal{A}^0_\mathcal{C}(E)(X)$.
\end{theorem}
\begin{proof}
The proof follows in the same line of argument as in \cite[Chapter 4, Section 4, theorem 4.9]{ROW}, using Proposition \ref{Sobolev} (cf. \cite[Chapter 4, Section 1, Proposition 1.1]{ROW}). 
\end{proof}
The Theorems \ref{theorem1} and \ref{theorem2} give the following theorem on finite dimension of Kernel of elliptic operators.
\begin{theorem}\cite[Chapter 4, Section 4, Theorem 4.8]{ROW}\label{findim}
Let $L\in \mathrm{Diff_m(E, F)}$ be an elliptic differential operator. Define $\mathcal{H}_{L_s}(X)=\text{Ker}\big(L_s:W^s_\mathcal{C}(E)(X)\rightarrow W^{s-m}_\mathcal{C}(F)(X)\big)$, then
\begin{enumerate}
\item $\mathcal{H}_{L_s}(X)\subset \mathcal{A}^0_\mathcal{C}(E)(X)$; hence $\mathcal{H}_{L_s}(X)=\mathcal{H}_L(X)$ for all $s\in \mathbb{Z}$.
\item $\text{dim}_\mathbb{R}(\mathcal{H}_{L_s}(X))=\text{dim}_\mathbb{R}(\mathcal{H}_L(X))<\infty$, $\text{dim}_\mathbb{R}\left(\bigslant{W^{s-m}_\mathcal{C}(F)}{L_s(W^s_\mathcal{C}(E))}\right)<\infty$.
\end{enumerate}
\end{theorem}
\begin{theorem}\cite[Chapter 4, Section 4, Theorem 4.11]{ROW}\label{ortho}
Let $L\in \mathrm{Diff}_m(E,F)$ be elliptic and suppose that $\beta\in \mathcal{H}_L^\perp(X)\cap \mathcal{A}^0_\mathcal{C}(F)(X)$. Then there exists a unique section $\alpha\in \mathcal{A}^0_\mathcal{C}(E)(X)$ such that $L(\alpha)=\beta$ and $\alpha$ is orthogonal to $\mathcal{H}_L(X)$ in $W^0_\mathcal{C}(E)(X)$.
\end{theorem}
\begin{definition}\rm{
A differential operator $L\in \mathrm{Diff}_m(E)$ is called self-adjoint if $L^\star=L$.
}
\end{definition}
Using previous Theorems \ref{theorem2}, \ref{findim} and \ref{ortho} and following the same line of arguments as in \cite[Chapter 4, Section 4, Theorem 4.12]{ROW}, we have a fundamental decomposition theorem for self-adjoint elliptic operators.
\begin{theorem}\cite[Chapter 4, Section 4, Theorem 4.12]{ROW}\label{orthodeco}
Let $L\in \mathrm{Diff}_m(E)$ be a self-adjoint elliptic operator. Then there exist elements $H_L,G_L\in \mathrm{Hom}_\mathcal{C}(E,E)$ such that,
\begin{enumerate}
\item $H_L(\mathcal{A}_\mathcal{C}^0(E)(X))=\mathcal{H}_L$ and $\mathrm{dim}(\mathcal{H}_L)<\infty$
\item $L\circ G_L+H_L=G_L\circ L+H_L=id|_{\mathcal{A}^0_\mathcal{C}(E)(X)}$
\item $H_L$ and $G_L$ are operators in $\mathrm{OP}_0(E)$, which can be extended to bounded operators on $W^0_\mathcal{C}(E)(X)$.
\item The decomposition $\mathcal{A}^0_\mathcal{C}(E)(X)=\mathcal{H}_L\oplus G_L\circ L(\mathcal{A}^0_\mathcal{C}(E)(X))=\mathcal{H}_L\oplus L\circ G_L(\mathcal{A}^0_\mathcal{C}(E)(X))$ is orthogonal with respect to an inner product on $W^0_\mathcal{C}(E)(X)$.
\end{enumerate}
\end{theorem}



\begin{thebibliography}{012345}

\bibitem{SAAJ} Amrutiya Sanjay and Ayush Jaiswal, \emph{On d-holomorphic connections}, 
Proceedings-Mathematical Sciences 133.2 (2023): 21.

\bibitem{MA}Atiyah, Michael Francis. \emph{Complex analytic connections in fibre bundles}, Transactions of the American Mathematical Society 85.1 (1957): 181-207.

\bibitem{IBNR} Biswas Indranil, and Nyshadham Raghavendra. \emph{The Atiyah-Weil criterion for 
holomorphic connections}, Indian J. pure appl. Math 39.1 (2008): 3-47.

\bibitem{EBJJEJMGGG}Bujalance, Emilio, Jos\'{e} J. Etayo, Jos\'{e} M. Gamboa, and Grzegorz Gromadzki. Automorphism groups of compact bordered Klein surfaces. Vol. 1439. Berlin: Springer-Verlag, 1990.

\bibitem{ZK} Khaled, Z. (2023), \emph{Real structures on primary Hopf surfaces}, arXiv preprint arXiv:2310.05265.

\bibitem{KK} Kodaira, K. (1966), \emph{On the structure of compact complex analytic surfaces, II}, American Journal of Mathematics, 88(3), 682-721.

\bibitem{KK1} Kodaira, K. (1966), \emph{Complex structures on $S^1 \times S^3$}, Proceedings of the National Academy of Sciences, 55(2), 240-243.

\bibitem{LH} H\"{o}rmander Lars, \emph{Linear Partial Differential Operators}, Springer-Verlag 
New York Inc., New York, and Academic Press, Inc., New York, 1963.

\bibitem{WR} Rudin Walter, \emph{Functional analysis}, India, McGraw-Hill, 1991.

\bibitem{SW1} Wang Shuguang, \emph{Twisted complex geometry}, Journal of the Australian Mathematical 
Society 80.2 (2006): 273-296.

\bibitem{ROW} Wells, Raymond O., \emph{Differential Analysis on Complex Manifolds (new appendix by 
Oscar-Garcia Prada)}. Vol. 65. Springer Science and Business Media, 2007.

\end{thebibliography}
\end{document}